\documentclass[12pt,reqno]{amsart}
\numberwithin{equation}{section}

\usepackage{graphicx} 
\usepackage[all,cmtip]{xy}
\usepackage{calligra}
\usepackage[T1]{fontenc}
\usepackage[mathscr]{eucal} 
\usepackage{lscape}
\usepackage{stmaryrd}
\usepackage{amsmath}
\usepackage{amsfonts}
\usepackage{amsthm}
\usepackage{mathrsfs}
\usepackage{amssymb}
\usepackage[all]{xy}
\usepackage{hyperref}
\newcommand{\Path}{\mathrm{Path}}
\newcommand{\Lie}{\mathrm{Lie}}
\newcommand{\Met}{\mathrm{Met}}
\newcommand{\Ker}{\mathrm{Ker}}
\newcommand{\Ad}{\mathrm{Ad}}
\newcommand{\ad}{\mathrm{ad}}
\newcommand{\id}{\mathrm{Id}}
\newcommand{\GL}{\mathrm{GL}}
\newcommand{\SL}{\mathrm{SL}}
\newcommand{\U}{\mathrm{U}}
\providecommand{\norm}[1]{\lVert#1\rVert}
\newcommand{\Id}{\mathrm{Id}}
\newcommand{\End}{\mathrm{End}}
\newcommand{\Hom}{\mathrm{Hom}}
\newcommand{\Aut}{\mathrm{Aut}}
\newcommand{\aut}{\mathrm{aut}}
\newcommand{\Conj}{\mathrm{Conj}}
\newcommand{\ssigma}{\pmb{\sigma}}
\newcommand{\sm}{\mathrm{ss}}
\newcommand{\irr}{\mathrm{irr}}

\newcommand{\K}{\mathrm{K}}
\newcommand{\hol}{\mathrm{hol}}
\newcommand{\Sym}{\mathrm{Sym}}

\newcommand{\Hh}{\mathcal{H}}
\newcommand{\Mm}{\mathcal{M}}
\newcommand{\Oo}{\mathcal{O}}
\newcommand{\Rr}{\mathcal{R}}

\newcommand{\CC}{\mathbb{C}}

\newcommand{\R}{\mathbb{R}}
\newcommand{\Z}{\mathbb{Z}}
\newcommand{\VV}{\mathbb{V}}
\newcommand{\EE}{\mathbb{E}}
\newcommand{\pP}{\mathfrak{p}}
\newcommand{\gG}{\mathfrak{g}}
\newcommand{\hH}{\mathfrak{h}}
\newcommand{\kK}{\mathfrak{k}}
\newcommand{\zZ}{\mathfrak{z}}
\newcommand{\lL}{\mathfrak{l}}
\newcommand{\uU}{\mathfrak{u}}
\newcommand{\mM}{\mathfrak{m}}
\newcommand{\lieg}{\mathfrak{g}}
\newcommand{\lieh}{\mathfrak{h}}
\newcommand{\liem}{\mathfrak{m}}
\newcommand{\AAA}{\curly{A}}
\newcommand{\CCC}{\curly{C}}

\newcommand{\GGG}{\curly{G}}
\newcommand{\HHH}{\curly{H}}
\newcommand{\KKK}{\curly{K}}

\newcommand{\SSS}{\curly{S}}
\newcommand{\TTT}{\curly{T}}
\newcommand{\XXX}{\curly{X}}
\newcommand{\PPP}{\curly{P}}

\newcommand{\Mg}{\mathcal{M}^{\operatorname{gauge}}}
\newcommand{\dbar}{\overline{\partial}}
\newcommand{\cM}{\mathcal{M}}
\newcommand{\calR}{\mathcal{R}}

\newcommand{\mapnormal}[5]{\begin{array}{cccc} #5:& #1 & \stackrel{ }{\longrightarrow} & #2 \\ \,\,\,\,&#3 & \longmapsto & #4 \end{array}}
\newcommand{\map}[5]{\begin{array}{ccc} #1 & \stackrel{#5}{\longrightarrow} & #2 \\ #3 & \longmapsto & #4 \end{array}}

\DeclareFontFamily{OT1}{rsfs}{}
 \DeclareFontShape{OT1}{rsfs}{n}{it}{<->rsfs10}{}
 \DeclareMathAlphabet{\curly}{OT1}{rsfs}{n}{it}

\newtheorem{The}{Theorem}[section]
\newtheorem{Lem}[The]{Lemma}
\newtheorem{Pro}[The]{Proposition}
\newtheorem{Cor}[The]{Corollary}
\newtheorem{Def}[The]{Definition}

\theoremstyle{remark}
\newtheorem{Rem}[The]{Remark}
\newtheorem{Exam}[The]{Example}

\begin{document}

\subjclass{53B35, 14F05}

\title[Real Higgs pairs and non-abelian Hodge correspondence]{Real Higgs pairs and 
non-abelian Hodge correspondence on a Klein surface}

\author[I. Biswas]{Indranil Biswas}

\address{Indranil Biswas, School of Mathematics, Tata Institute of Fundamental
Research, Homi Bhabha Road, Mumbai 400005, India}

\email{indranil@math.tifr.res.in}

\author[L. A. Calvo]{Luis \'Angel Calvo}

\address{Luis Angel Calvo, Universidad Pontificia de Comillas, ICAI, Alberto Aguilera, 25, 28015 Madrid, Spain}

\email{lacalvo@icai.comillas.edu}

\author[O. Garc\'{\i}a-Prada]{Oscar Garc\'{\i}a-Prada}

\address{Oscar Garc\'ia-Prada, Instituto de Ciencias Matem\'aticas, CSIC-UAM-UC3M-UCM, Nicol\'as
Cabrera, 13--15, Campus Cantoblanco, 28049 Madrid, Spain}

\email{oscar.garcia-prada@icmat.es}

\keywords{$G$-Higgs bundle, Hermite--Einstein--Higgs equation, non-abelian Hodge theory, real 
bundle, anti-holomorphic involution.}

\subjclass[2010]{14H60, 14D20, 14H52} 

\date{30 September 2020}

\begin{abstract}
We introduce real structures on $L$-twisted Higgs pairs over a compact connected Riemann surface $X$
equipped with an anti-holomorphic involution, where $L$ is a holomorphic line bundle on $X$ with a real
structure, and prove a Hitchin--Kobayashi correspondence for the $L$-twisted Higgs pairs.
Real $G^\R$-Higgs bundles, where $G^\R$ is a real form of a connected semisimple complex
affine algebraic group $G$, constitute
a particular class of examples of these pairs. In this case,
the real structure of the moduli space of $G$-Higgs pairs is defined using
a conjugation of $G$ that commutes with the one defining the real form $G^\R$ and a compact conjugation of $G$
preserving $G^\R$. We establish a
homeomorphism between the moduli space of real $G^\R$-Higgs bundles and the moduli space of 
representations of the fundamental group of $X$ in $G^\R$ that can be extended to a representation of the orbifold 
fundamental group of $X$ into a certain enlargement of $G^\R$ with quotient $\Z/2\Z$. Finally, we show how 
real $G^\R$-Higgs bundles appear naturally as 
fixed points of certain anti-holomorphic involutions of the moduli space of $G^\R$-Higgs bundles,
constructed using the real structures on $G$ and $X$. A similar result is proved 
for the representations of the orbifold fundamental group.
 \end{abstract}
\maketitle

\tableofcontents

\section{Introduction}\label{Section1}

In recent years, much attention has been paid to the theory of Higgs pairs. The study of these 
objects is primarily motivated by various moduli problems arising from gauge theory, algebraic 
geometry, symplectic geometry and topology. To recall the definition of a Higgs pair, let $X$ be a 
compact connected Riemann surface, $G$ a connected reductive complex affine algebraic group, $\VV$ 
a complex vector space, $\rho:G\longrightarrow \GL(\VV)$ a holomorphic representation and $L$ a 
holomorphic line bundle over $X$. A $L$-\textbf{twisted Higgs pair} $(E,\,\varphi)\,$\textbf{ of 
type} $\rho$ is a pair consisting of a holomorphic principal $G$-bundle $E$ over $X$ and a 
holomorphic section $\varphi$ of $V\otimes L$, where $V\,:=\,E(\VV)$ is the holomorphic vector bundle over $X$ 
associated to $E$ via $\rho$. In \cite{GGM:21}, the notion of $\alpha$-polystability for these 
pairs was introduced, where $\alpha$ is an element of the center of the Lie algebra $\kK$ of a
fixed maximal compact subgroup $K\,\subset\, G$. The \textbf{Hitchin-Kobayashi correspondence} in this 
context, also proved in \cite{GGM:21}, has the following formulation:

Once we fix a K\"ahler form $\omega$ of $X$, a $L$-twisted Higgs pair $(E,\varphi)$ of type $\rho$ 
is $\alpha$-polystable if and only if there is a reduction $h$ of structure group of the principal 
$G$-bundle $E$, to the maximal compact subgroup $$K\, \subset\, G\, ,$$ that satisfies the 
Hermite--Einstein--Higgs equation
\begin{equation}\label{eq:1}
\Lambda\,F_h+\mu_h(\varphi)\,=\,-\sqrt{-1}\,\alpha\,,
\end{equation} 
where $F_h$ is the curvature of the unique connection, on the principal $K$-bundle $E_K\,\subset\, 
E$ corresponding to the reduction $h$, which is compatible with the holomorphic structure of $E$, 
while $\Lambda$ denotes the contraction of differential forms on $X$ using the K\"ahler form
$\omega$, and $\mu_h$ is a moment map that depends on $h$.

The construction of the above mentioned moment map $\mu_h$ requires fixing a 
bi-invariant inner product $B$ on the Lie algebra $\kK\, =\, \text{Lie}(K)$, a $K$-invariant Hermitian 
product $\langle\, ,\,\rangle$ on $\VV$ as well as a Hermitian metric $h_{L}$ on $L$. Examples of 
twisted Higgs pairs are quiver bundles, Higgs bundles and Hodge bundles among other objects (see 
\cite{LAC, bradlowtriples, bradlowtriples1, GGM:21, GPR} and references therein for examples and more on Higgs pairs).

One of the main goals of this paper is to extend the above correspondence to the context where all
the objects are equipped with real structures. By these we mean
anti-holomorphic involutions $\sigma_{X}$ and $\sigma_{G}$ on $X$ and $G$, respectively, an
anti-linear involution $\sigma_{\VV}$ on $\VV$ such that the holomorphic representation $\rho$ is
compatible with respect to $\sigma_{G}$ and $\sigma_{\VV}$, as well as an anti-holomorphic
involution $\sigma_{L}$ on $L$ over $\sigma_X$, which is anti-linear on the fibers. 
A $(\sigma_{X},\sigma_{G},c)$-real structure on a holomorphic principal $G$-bundle
$E$ over $X$ is an anti-holomorphic automorphism $\sigma_{E}$ of $E$ over
the involution $\sigma_X$ of $X$, such that 
$\sigma_{E}(eg)\,=\,\sigma_{E}(e)\, \sigma_{G}(g)$ and $\sigma^2_E(e)\,=\,e\,c$, for all
$e\,\in\, E$ and $g\,\in\, G$, where $c\,\in\, Z^{\sigma_G}_{2}\cap \Ker(\rho)$ with $Z^{\sigma_G}_{2}$
being the subgroup
of the center $Z(G)$ of $G$ consisting of all the elements of order two invariant under
the involution $\sigma_G$.
With these real structures in place, set $$\ssigma\,=\,(\sigma_{X},\,\sigma_{G},\,c,\, \sigma_{L},\, \sigma_{\VV},\,\pm)\, .$$
A \textbf{$\ssigma$-real $L$-twisted Higgs pair of type $\rho$} is a triple of the form
$(E,\,\varphi, \,\sigma_E)$, where
\begin{itemize}
\item $(E,\,\varphi)$ is a $L$-twisted Higgs pair of type $\rho$,

\item $\sigma_E$ is a $(\sigma_{X},\sigma_{G},c)$-real structure
on $E$, and

\item $\varphi$ satisfies the condition $(\sigma_{V}\otimes\sigma_{L})(\varphi)\,=\,
\pm\sigma_{X}^{*}\varphi$, with $\sigma_{V}$ being the involution of
$V\,=\,E(\VV)$ induced by $\sigma_{E}$ and $\sigma_{\VV}$.
\end{itemize}

A reduction of structure group $E_{K} \, \subset\, E$ 
of $E$ to the maximal compact subgroup $K$ is said to be $\sigma_{E}$-compatible 
if it is preserved by the self-map $\sigma_{E}$ of $E$.
We prove the following Hitchin--Kobayashi correspondence (see Theorem 
\ref{th:Hitchin--Kobayashi}):

\begin{The}\label{thmi}
A $\ssigma$-real $L$-twisted Higgs pair $(E,\,\varphi,\,\sigma_{E})$ of type $\rho$ is
polystable if and only if there is a $\sigma_{E}$-compatible reduction $h$ of the
structure group of $E$, from $G$ to $K$, that satisfies equation \eqref{eq:1}.
\end{The}

In order to prove Theorem \ref{thmi}, we first adapt the arguments in \cite{GGM:21} to reduce the 
proof to the case of stable Higgs pairs. The main difficulty here is to characterize the space of infinitesimal 
automorphisms for a stable $\ssigma$-real $L$-twisted Higgs pair $(E,\,\varphi,\,\sigma_{E})$, taking 
into account our stability condition that involves a condition on the adjoint bundle $\Ad(E)$. 
Once that is achieved, we adapt the arguments in \cite{BGMR:23} to prove the above theorem in the stable 
case. This approach involves finding a minimizing sequence of $\sigma_{E}$-compatible metrics for 
the integral of the moment map (\ref{eq:1}), converging weakly to the 
solution we are seeking.

Now, let $G^\R$ be a real form of a connected semisimple complex affine algebraic group $G$. 
Let $H^\R\subset G^\R$ be a maximal compact subgroup, and $\lieg^\R=\lieh^\R\oplus \liem^\R$ be 
the Cartan decomposition of $\lieg^\R$, the Lie algebra of $G^\R$, where $\lieh^\R$ is the Lie algebra of 
$H^\R$ and $\liem^\R$ is its orthogonal complement with respect to the Killing form
on $\lieg^\R$. Let $H$ and $\liem$ be the 
complexifications of $H^\R$ and $\liem^\R$ respectively.

A particular class of $L$-twisted Higgs pairs are the \textbf{$G^\R$-Higgs bundles}. For a $G^\R$-Higgs 
bundle, the above line bundle $L$ is the holomorphic cotangent bundle $K_X$ of $X$, the structure group 
is $H$, while $\rho$ is the adjoint representation $\iota\,:\, 
H\,\longrightarrow \,\GL(\liem)$. In other words, a $G^\R$-Higgs bundle is a pair $(E,\varphi)$, where $E$ is a
holomorphic principal $H$-bundle on $X$ and $\varphi$ is a holomorphic section of the
holomorphic vector bundle $E(\liem)\otimes K_X$. We recall that the non-abelian 
Hodge correspondence for $G^\R$-Higgs bundles produces a homeomorphism between the 
moduli space $\Mm(G^\R)$ of polystable $G^\R$-Higgs bundles on $X$ and the character variety $\Rr(G^\R)$ of equivalence 
classes of reductive representations of the fundamental group of $X$ in $G^\R$ (see \cite{GGM:21}).

Another main result of this paper is the extension of the non-abelian Hodge correspondence to the 
context of \textbf{real $G^\R$-Higgs bundles}, where $G^\R$ is, as above, a real form of a connected 
semisimple complex affine algebraic group $G$. For this set-up, we need to consider a real 
structure $\sigma_{G}$ on $G$ which satisfies the condition that it commutes with a fixed compact 
conjugation $\tau$ and the conjugation $\mu$ of $G$ defining the real form
 $G^\R$. Then 
$\sigma_{G}$ preserves $H$ while $d\sigma_{G}$ preserves $\mM$. Moreover, $\iota$ is a 
representation compatible with $\sigma_{G}$ and $d\sigma_{G}$. Real $G^\R$-Higgs bundles are 
$\ssigma$-real $K_X$-twisted Higgs pairs of type $\iota$, where $\sigma_{\VV}\,=\,d\sigma_{G}$, 
while $\sigma_{L}\,=\,\sigma_{K_X}$ is the natural anti-holomorphic involution of $L\,=\,K_{X}$ 
induced by $\sigma_{X}$. Note that the case of real $G^\R$-Higgs bundles generalizes the case of real Higgs bundles for a complex group 
studied in \cite{BGH:98}, \cite{BGH:99} and \cite{BH:11}, since a complex semisimple Lie group $G$ 
can be viewed as a real form of $G\times\overline{G}$, where the anti-holomorphic involution is
$(x,\, y)\, \longmapsto\, (y,\, x)$.

To describe the other side of the non-abelian Hodge correspondence,
let $\Gamma(X,\,x)$ be the orbifold fundamental group of $(X,\,\sigma_X)$ for a base point $x\,\in\, X$. Let 
$$c\,\in\, Z^{\sigma_{G}}_{2}(H)\cap \Ker(\iota)\cap Z(G^\R)\, ,$$ where $Z(H)$ and 
$Z(G^\R)$ are the centers of $H$ and $G^\R$ respectively, and $Z^{\sigma_{G}}_{2}(H)$ is the subgroup of elements of order 2 in $Z(H)$ invariant under $\sigma_G$.
Also, let
$\widehat{G}^\R_\pm=\widehat{G}^\R_\pm(\sigma_G,c)$ be the group whose underlying set is $G^\R\times(\mathbb{Z}/2\mathbb{Z})$ and
the group operation is given by the rule
$$(g_1,\,e_1)(g_2,\,e_2)\,=\,(g_{1}(\sigma_{G}\tau^{\frac{1}{2}\mp \frac{1}{2}})^{e_1}(g_2) c^{e_{1}+ e_{2}},
\,e_{1}+e_{2})\, .$$

A representation $\widehat{\rho}\,:\,
\Gamma(X,\, x)\,\longrightarrow\, \widehat{G}^\R_\pm$ 
is called $(\sigma_X,\sigma_{G},c,\pm)$-\textbf{compatible} 
if it is an extension of a representation 
$\rho\,:\, \pi_1(X,\, x)\,\longrightarrow\, G^\R$ 
fitting in a commutative diagram
of homomorphisms
\begin{equation}
\xymatrix{
0\ar[r] &\pi_1(X,\,x)\ar[r]^{i}\ar[d]^{\rho} &\Gamma(X,\, x)\ar[r]^{q}\ar[d]^{\widehat{\rho}} &\mathbb{Z}/2\mathbb{Z}\ar[r]\ar[d]^{\Id}&0\\
0\ar[r] &G^\R\ar[r]^{i'} &\widehat{G}^\R_\pm\ar[r]^{q'} &\mathbb{Z}/2\mathbb{Z}\ar[r]&0\,,
}
\end{equation}
where $i$ and $i'$ are the inclusion maps and $q$ and $q'$ are the corresponding
projections.

Let $\Rr(G^\R,\sigma_X,\sigma_{G},c,\pm)$ be the variety consisting of
$G^\R$-conjugacy classes of $(\sigma_X,\sigma_{G},c,\pm)$-compatible 
representations 
$\widehat{\rho}\,:\,
\Gamma(X,\, x)\,\longrightarrow\, \widehat{G}^\R_\pm$ 
whose restriction to $\pi_1(X,\,x)$ is reductive, that is,
its conjugacy class is an element in $\calR(G^\R)$.

We prove the following (see Theorem \ref{th:non-abelian}):

\begin{The}\label{thmi2}
There is a canonical homeomorphism
between $\Rr(G^\R,\sigma_X,\sigma_{G},c,\pm)$
and the moduli space $\Mm(G^\R,\sigma_{X},\sigma_{G},c,\pm)$ of polystable $(\sigma_{X},\sigma_{G},c,\pm)$-real $G^\R$-Higgs bundles.
\end{The}

The proof of Theorem \ref{thmi2} crucially uses Theorem 
\ref{thmi} and an appropriate version of the Donaldson--Corlette
theorem on the existence of harmonic metrics. It may be mentioned that Theorem \ref{thmi2} 
could be the starting point to identify higher Teichm\"uller spaces in this real context, as it is done in the 
usual theory of $G^\R$-Higgs bundles (see \cite{prada} for a review).

Real $G^\R$-Higgs bundles appear in a natural way as fixed points of the involutions of the moduli space $\Mm(G^\R)$
of $G^\R$-Higgs bundles defined by
\begin{equation}\label{eq:ecuacionesp1}
\mapnormal{\Mm(G^\R)}{\Mm(G^\R)}{(E,\,\varphi)}{(\sigma_{X}^{*}\sigma_{G}E,\pm\sigma_{X}^{*}\sigma_{G}
\varphi)\,.}{\iota_{\Mm}(\sigma_{X},\,\sigma_{G})^{\pm}}
\end{equation}
 
The obvious forgetful map induces a map from $\Mm(G^\R,\sigma_{X},\sigma_{G},c,\pm)$ 
to $\Mm(G^\R)$, and we prove the following (see Proposition \ref{pro:pseudoreales}):

\begin{Pro}\label{propi}
The image of $\Mm(G^\R,\sigma_{X},\sigma_{G},c,\pm)$ in $\Mm(G^\R)$ is contained in the fixed point set 
$\Mm(G^\R)^{\iota_{\Mm}(\sigma_{X},\sigma_{G})^{\pm}}$. Furthermore, if we restrict the
involution $\iota_{\Mm}(\sigma_{X},\sigma_{G})^{\pm}$
in \eqref{eq:ecuacionesp1} to the subvariety of $\Mm_{\sm}(G^\R)\subset \Mm(G^\R)$, consisting of stable and 
simple $G^\R$-Higgs bundles,
then $\Mm_\sm(G^\R)^{\iota_{\Mm}(\sigma_{X},\sigma_{G})^{\pm}}$ is contained in the image in $\Mm(G^\R)$ of
the union of moduli spaces $\Mm(G^\R,\sigma_{X},\sigma_{G},c,\pm)$
parameterized by elements $c\,\in\, Z^{\sigma_{G}}_{2}(H)\cap \Ker(\iota)$, where $Z(H)$ is the
center of $H$ and $Z^{\sigma_{G}}_{2}(H)$ is the subgroup of elements of order two in $Z(H)$ invariant under $\sigma_G$.
\end{Pro}

We also prove a similar result for $\Rr(G^\R,\sigma_X,\sigma_{G},c,\pm)$, where the involutions of the
character variety $\Rr(G^\R)$ of the fundamental group of $X$ in $G^\R$ are given by
\begin{equation}\label{eq:ecuacionesr}
 \mapnormal{\Rr(G^\R)}{\Rr(G^\R)}{\rho}{\sigma_{G} \circ \tau^{\frac{1}{2}\mp\frac{1}{2}}\circ\rho \circ
(\sigma_{X})_{*}\,.}{\iota_{\Rr}(\sigma_{X},\sigma_{G})^{\pm}}
\end{equation}
(See Proposition \ref{rep-involutions}.)

When the group $G^\R$ is complex, the moduli space $\Mm(G^\R)$ is a hyper-K\"ahler manifold and the fixed point sets
of the involutions in 
(\ref{eq:ecuacionesp1}) and (\ref{eq:ecuacionesr}) define branes, in the sense of \cite{KapustinWitten}. These branes 
have been studied in \cite{Baraglia}, \cite {elliptic}, \cite{BGP}, \cite{francoana}, \cite{GPR},
\cite{garcia-prada&wilkin} and \cite{BGH3b}. We note that their importance stems from their close relation with 
mirror symmetry and the Langlands correspondence. 

In some sense, in this paper, we are considering {\em doubly real} Higgs bundles since we are studying real 
structures on Higgs bundles whose structure group is already real. This is perfectly possible since, although the 
group $G^\R$ is a real form of a complex group $G$, $G^\R$-Higgs 
bundles are holomorphic objects on $X$ even if $G^\R$ does not have a complex structure,
and hence the reality conditions can be defined by choosing another conjugation 
of $G$ preserving $G^\R$ as well as a conjugation of $X$. As we have explained above, in the paper we go beyond Higgs 
bundles, studying real structures on more general Higgs pairs.

The article is organized as follows. First, we review in Section \ref{Section2} the notions of real structures for 
the main objects that will be used in the subsequent sections, and we also analyze the Chern correspondence between 
holomorphic structures and connections in the presence of real structures. In Sections \ref{Section3} and 
\ref{sec4n} we prove a Hitchin--Kobayashi correspondence for $\ssigma$-real $L$-twisted Higgs pairs (here we follow 
\cite{GGM:21} closely). Firstly, after defining these pairs, we give some examples and recall the 
Hermite--Einstein--Higgs equation. We then prove that a polystable $\ssigma$-real $L$-twisted Higgs pair that is 
not stable admits a Jordan--H\"older reduction and, as a consequence of it, we reduce the proof to the stable case. 
Following that, we reformulate our problem in terms of finding a metric on which the integral of the moment map attains a 
minimum. Finally, we prove the converse which says that the existence of solutions of the Hermite--Einstein--Higgs 
equation implies polystability. In Section \ref{Section4}, we prove the non-abelian Hodge correspondence for real 
$G^\R$-Higgs bundles, where $G^\R$ is a real form of a connected semisimple complex Lie group $G$. We start by 
proving a bijective correspondence between $\Mm(G^\R,\sigma_{X},\sigma_{G},c,\pm)$ and the moduli space of triples 
consisting of a real structure on a $C^{\infty}$ principal $H^\R$-bundle, a compatible connection, and a real Higgs field 
satisfying the Hitchin equations. We also prove an appropriate Donaldson--Corlette correspondence, showing the 
bijection between the moduli space of solutions to Hitchin equations and the moduli space of compatible reductive 
flat $G^\R$-connections, which in turn is in bijective correspondence with $\Rr(G^\R,\sigma_X,\sigma_{G},c,\pm)$. 
Finally, in Section \ref{Section5}, we describe the relationship between the fixed-point of the involutions in 
\eqref{eq:ecuacionesp1} and the moduli spaces $\Mm(G^\R,\sigma_{X},\sigma_{G},c,\pm)$, where $c$ varies in 
$Z^{\sigma_{G}}_{2}(H)\cap \Ker(\iota)$. By the non-abelian Hodge correspondence, proved in the previous section, 
with the additional condition that $c\in Z^{\sigma_{G}}_{2}(H)\cap \Ker(\iota)\cap Z(G^\R)$, we obtain a similar 
description for the fixed points of the involutions in \eqref{eq:ecuacionesr} in terms of the moduli spaces 
$\Rr(G^\R,\sigma_X,\sigma_{G},c,\pm)$.

\section{Real structures}\label{Section2}

\subsection{Real structures on complex Lie groups and their representations}

We recall that a {\bf real structure} or {\bf conjugation} on a reductive complex affine algebraic group $G$ 
is an anti-holomorphic involution of $G.$ We denote by $\Conj(G)$ the set of conjugations on 
$G$, and also denote by $\Aut_{2}(G)$ the set of holomorphic
automorphisms of $G$ of order two (and the identity map). In $\Aut_{2}(G)$ and 
$\Conj(G)$ we define the equivalence relation $ \sigma \,{\sim}\,\sigma'$ if and only if there 
is some $\alpha\,\in\, \Aut(G)$ with the property that $\sigma'\,=\,\alpha^{-1}\circ \sigma \circ \alpha$. 
Cartan proved in \cite{Cartan} that there is a bijection
\begin{equation}\label{CT}
 \map{\Conj(G)/\!\!\sim}{\Aut_2(G)/\!\!\sim\,}{\left[\sigma\right]}{\left[\theta\right]\,,}{ }
\end{equation}
which is constructed as follows: Fix a compact conjugation $\tau$ of $G$; now given any
$$\sigma \,\in\, \Conj(G)\, ,$$ there is an element
$\sigma'\in \Conj(G)$ such that $\sigma'\,\sim\,\sigma$ and $\theta\,:=\,\sigma'\circ\tau\,\in\, \Aut_2(G)$.

\begin{Pro}\label{max}
Given any $\sigma\,\in\, \Conj(G),$ there is a maximal compact subgroup $$K\,\subset\, G$$ such that $\sigma(K)\,=\,K.$
\end{Pro}

\begin{proof}
Given a compact conjugation $\tau$, by the bijection in (\ref{CT}) there is a conjugation
$$\sigma'\,:=\,\alpha^{-1}\circ \sigma \circ \alpha\, ,$$ for some $\alpha\in \Aut(G)$, such that $\sigma'\circ\tau\,=\,
(\sigma'\circ\tau)^{-1}$, and therefore we have
\begin{equation}\label{eq:commut}
\tau\circ\sigma'\,=\,\sigma'\circ\tau\, ,
\end{equation}
because $\tau$ and $\sigma'$ are both involutions. 
Consider $\tau'\,:=\,\alpha\circ \tau \circ \alpha^{-1}$; then $K\,:=\,G^{\tau'}$ is a maximal compact subgroup of $G$.
For any $x\, \in\, G^{\tau'}$, we have
$$
\tau'(\sigma(x)) \,=\, \alpha\circ \tau \circ \alpha^{-1}\circ \alpha\circ \sigma' \circ \alpha^{-1}(x)
\,=\, \alpha\circ \tau \circ\sigma'\circ\alpha^{-1}(x)
$$
$$
=\, \alpha\circ\sigma'\circ\alpha^{-1}\circ\alpha\circ\tau
\circ\alpha^{-1}(x)\,=\, \sigma(\tau'(x))
$$
using Equation (\ref{eq:commut}). This implies that $\sigma(G^{\tau'})\,=\, G^{\tau'}$.

An alternative proof: Using the involution $\sigma$ of $G$, construct the semi-direct
product $G\rtimes ({\mathbb Z}/2{\mathbb Z})$. Let $\widetilde K$ be a maximal compact
subgroup of $G\rtimes ({\mathbb Z}/2{\mathbb Z})$. Then $\widetilde{K}\cap G$ is a
maximal compact subgroup of $G$ which is preserved by $\sigma$.
\end{proof}

Let $G$ be a complex reductive affine algebraic group equipped with a real structure $\sigma_G$, and let 
$\VV$ be a complex vector space equipped with a real structure 
$\sigma_{\mathbb{V}}\,:\,\VV\,\longrightarrow\, \VV$, that is, an anti-linear involution. A 
holomorphic representation $\rho\,:\, G\,\longrightarrow\, \GL(\mathbb{V})$ is 
$(\sigma_G,\,\sigma_{\VV})$-\textbf{compatible} if for every $g\,\in\, G$ and $v\,\in\,\mathbb{V},$ we have
$$\sigma_\mathbb{V}(\,\rho(g)(v))\,=\,\rho(\sigma_G(g))(\sigma_{\mathbb{V}}(v))\,.$$

\begin{Exam}[{The adjoint representation}]\label{adjointrep}
Let $G$ be a complex group equipped with a real structure $\sigma_G$. Then the corresponding 
$\mathbb R$--linear automorphism $d\sigma_G\, :\, \gG\, \longrightarrow\, \gG$ 
is a real structure on the Lie algebra $\gG$ of $G$. Since the adjoint action $\Ad\,:\,G\,\longrightarrow 
\,\Aut(G)$ commutes with $\sigma_G$, it follows that $\Ad\,:\,G\,\longrightarrow\, \GL(\gG)$ is 
a $(\sigma_G,\,d\sigma_G)$-compatible representation.
\end{Exam}

\begin{Exam}[The isotropy representation]\label{isotropyrep}
Let $G^\R$ be a real form of a connected reductive Lie group $G$, and let
$\tau \,\in\, \Conj(G)$ be a compact conjugation. Denote $G^\tau$ by $K$. Let
$\mu\,\in \,\Conj(G)$ be a conjugation such that
the fixed point set $(G)^{\mu}$ is $G^\R$. A conjugation $\sigma\,\in\, \Conj(G)$ is
called $(\mu,\,\tau)$-\textbf{compatible} if
\begin{enumerate}
\item $\sigma\circ\mu\,=\,\mu\circ \sigma\,$, and
\item $\sigma\circ\tau\,=\,\tau\circ \sigma\,.$
\end{enumerate}
If $\sigma$ is a $(\mu,\,\tau)$-compatible conjugation, then $H^\R\,=\,G^\R\cap K$ is a maximal compact 
subgroup of $G^\R$ , whose complexification $H$ is preserved by $\sigma$.
Moreover, the Cartan decomposition
$\gG^\R\,=\, \hH^\R\,\oplus\,\mM^\R$ and its complexification
$\gG\,=\, \hH\,\oplus\,\mM$ are preserved by $d\sigma$. 
The adjoint action of $H^\R$ on $\mM^\R$ (the isotropy representation) 
extends to the complexification, giving a representation 
$\iota\,:\,H\,\longrightarrow\, \GL(\mM)\,$ which is $(\sigma,\,d\sigma)$-compatible.
\end{Exam}

\subsection{Real structures on Riemann surfaces and holomorphic bundles}\label{real-riemann}

We recall that a \textbf{real structure} on a compact Riemann surface $X$ is an anti-holomorphic
involution $$\sigma_X\,:\, X\,\longrightarrow\, X\, .$$
The pair $(X,\,\sigma_X)$ sometimes will be referred to as a \textbf{Klein surface}. Such pairs were first studied in \cite{K} and \cite{Wei}.
A \textbf{morphism of Klein surfaces}
$$f\,:\,(X,\,\sigma_X)\,\longrightarrow\, (X',\,\sigma_{X'})$$
is a holomorphic map $f\,:\,X\,\longrightarrow\, X'\,,$ such that 
$f\circ\sigma_{X}\,=\,\sigma_{X'}\circ f$. A K\"ahler form $\omega$ on a compact Klein surface 
$(X,\,\sigma_{X})$ is called \textbf{real} if $\sigma_{X}^{*}\omega\,=\,-\omega$. If 
$\omega$ is a K\"ahler form on $X$, then $\omega-\sigma_{X}^{*}\omega$ is a real K\"ahler form 
on $(X,\, \sigma_X)$ (see \cite[p. 4]{BGH:98}).

Let $(X,\,\sigma_{X})$ be a Klein surface, and let $V\,\longrightarrow \,X$ be a holomorphic vector bundle. 
A $\sigma_{X}$-\textbf{real structure} on $V$ is a $C^\infty$ isomorphism
$\sigma_{V}\,:\,V\,\longrightarrow\, V$ such that the following conditions hold:
\begin{enumerate}
\item $\sigma_{V}$ lifts $\sigma_{X}$, meaning that the diagram
$$
\xymatrix{
V\ar[d]\ar[r]^{\sigma_V}&V\ar[d]\\
X\ar[r]^{\sigma_{X}}&X}
$$
is commutative,

\item $\sigma_{V}$ is anti-holomorphic,

\item $\sigma_{V}$ is $\mathbb C$-anti-linear on the fibers of $V$, and

\item $\sigma_V^{2}\,=\,\id_{V}.$ 
\end{enumerate}

Such a pair $(V,\,\sigma_{V})$ will sometimes be referred to as a real holomorphic vector bundle. 
A homomorphism of real holomorphic vector bundles
$$f\,:\,(V,\,\sigma_{V})\,\longrightarrow\, (V',\,\sigma_{V'})$$ is an
${\mathcal O}_X$--linear homomorphism $f\,:\,V\,\longrightarrow\, V'$ satisfying the condition that
$$f\circ \sigma_{V}\,=\,\sigma_{V'}\circ f\, .$$ Real holomorphic vector bundles were introduced by 
Atiyah in \cite{Atiyah}, and they were subsequently studied in \cite {elliptic}, \cite{BGP}, \cite{BGH3}, \cite{LS:00}, \cite{SCH1} and \cite{SCH:122}.
The topological classes of these real bundles are described in \cite[Section 4]{BGH3}.

Let $G$ be a reductive complex affine algebraic group equipped with a real structure $\sigma_{G},$ and let $(X,\,\sigma_{X})$ be
a Klein surface. The center of $G$ will be denoted by $Z$.
Let $$Z_{2}^{\sigma_{G}}\, \subset\, Z$$ be the subgroup of $Z$ consisting of elements of $Z$ of order two 
(and the identity element) that are invariant under the involution
$\sigma_{G}$.
Take any $$c\,\in\, Z_{2}^{\sigma_{G}}\, .$$ Let $E$ be a holomorphic principal $G$-bundle over $X$. 

A $(\sigma_X,\sigma_G,c)$-\textbf{real structure} is a $C^\infty$ diffeomorphism
$\sigma_{E}\,:\,E\,\longrightarrow\, E$ such that the following four conditions holds:
\begin{enumerate}
\item the diagram
$$
\xymatrix{
E\ar[d]\ar[r]^{\sigma_E}&E\ar[d]\\
X\ar[r]^{\sigma_{X}}&X,}
$$
is commutative, or in other words $\sigma_{E}$ is a lift of $\sigma_{X}$,

\item $\sigma_{E}$ is anti-holomorphic,

\item $\sigma_{E}(eg)\,=\,\sigma_{E}(e)\,\sigma_{G}(g),$ for all $e\,\in\, E,\,g\,\in\, G\,,$ and

\item $\sigma_{E}^{2}(e)\,=\,ce\,.$
\end{enumerate}

A pair $(E,\,\sigma_{E})$ satisfying the above four conditions will
sometimes be referred to as a $(\sigma_X,\sigma_G,c)$-{\bf real holomorphic $G$-bundle}.

A morphism of $(\sigma_X,\sigma_G,c)$--real holomorphic $G$-bundles 
$$f\,:\,(E,\,\sigma_{E})\,\longrightarrow\, (E',\,\sigma_{E'})$$ is a holomorphic isomorphism of 
principal $G$-bundles $f\,:\,E\,\longrightarrow\, E'$ such that $$f\circ 
\sigma_{E}\,=\,\sigma_{E'}\circ f\,.$$ In the literature, such bundles are known as pseudo-real principal 
bundles (see \cite{BGH:98}, \cite{BGH:99} and \cite{BH:11}). In \cite[Theorem 3.9]{BGH:99}, 
there is a description of the topological classes of pseudo-real principal bundles.

\subsection{Chern correspondence and real structures}\label{subsection2.3}

It will be of importance for us to think of a holomorphic
principal bundle as a $C^\infty$ principal bundle equipped with a complex structure.
To explain this in detail, let $G$
be a reductive complex affine algebraic group equipped with a real structure $\sigma_{G}$, and let 
$(X,\,\sigma_{X})$ be a Klein surface. Take any $c\,\in\, Z_{2}^{\sigma_{G}}$, and let 
$\pi\,:\,\EE\,\longrightarrow\, X$ be a $C^\infty$ principal $G$-bundle over $X$. A 
$(\sigma_{X},\sigma_{G},c)$-{\bf real structure} $\sigma_{\EE}$ on $\EE$ is a lift of $\sigma_{X}$ to $\EE$
$$
\sigma_{\EE}\, :\, \EE\,\longrightarrow\, \EE
$$
satisfying the
conditions (1), (3) and (4) in the definition of a real structure of a holomorphic $G$-bundle given above. 
An {\bf almost complex structure} on $\EE$ is a $G$-invariant 
smooth section $J$ of $\End(T\EE)$, where $T\EE$ denotes the real tangent bundle of the total
space of $\EE$, such that
\begin{itemize}
\item $J^{2}\,=\,-\id\,,$ 

\item $d\pi \circ J\,=\,J_{X}\circ d\pi$, where $J_{X}$ is the complex structure of $X$, and

\item the map $\EE\times G\, \longrightarrow\, \EE$ giving the action of $G$ on $\EE$ is almost
holomorphic. This means that the differential of this map
$\EE\times G\, \longrightarrow\, \EE$
commutes with the almost complex structures; the almost
complex structure of $\EE\times G$ is given by $J$ and the almost complex structure of $G$.
\end{itemize}
It can be shown that an almost complex structure on $\EE$ is automatically a {\bf complex structure}.
Indeed, the obstruction to integrability of an almost complex structure on $\EE$ is an element
of $\Omega^{0,2}_X(\text{ad}(\EE))$, and it vanishes identically due to the fact that $\Omega^{0,2}_X\,=\, 0$
as the complex dimension of $X$ is one. Let $\CCC(\EE)$ be the space of all complex structures on
$\EE$. A complex structure $J_\EE\,\in\, \CCC(\EE)$ is called $\sigma_{\EE}$-\textbf{real} if
\begin{equation}\label{pr:cond1}
J_{\EE}\circ d\sigma_{\EE}\,=\,-d\sigma_{\EE}\circ J_{\EE}\, .
\end{equation}
We shall denote by $\CCC(\EE,\sigma_{\EE})$ the space of all complex structures on $\EE$ 
that are $\sigma_{\EE}$-real. The following proposition is a straight-forward fact.
 
\begin{Pro}\label{equivalenciaprincipal112}
There is a bijection between $\CCC(\EE,\sigma_{\EE})$ and the space of all equivalence classes of
$(\sigma_{X},\sigma_{G},c)$-real holomorphic principal $G$-bundles $(E,\,\sigma_{E})$ 
whose underlying smooth principal $G$-bundle is $\EE$, such that
the $C^\infty$ isomorphism between $E$ and $\EE$ takes $\sigma_E$ to $\sigma_{\EE}$. Two such
$(\sigma_{X},\sigma_{G},c)$-real holomorphic principal $G$-bundles $(E,\,\sigma_{E})$ and
$(E',\,\sigma'_{E})$ are called equivalent if the $C^\infty$ isomorphism of
principal $G$-bundles $E\, \longrightarrow\, E'$ given by the identity map of $\EE$ is holomorphic and
takes $\sigma_{E}$ to $\sigma'_{E}$.
\end{Pro}

Complex structures on $\EE$ are closely related to connections on $\EE$.
Let $\pi\, :\, \EE\, \longrightarrow\, X$ be a $C^\infty$ principal $G$-bundle
(here $G$ could be any Lie group). Recall that a {\bf connection} on $\EE$ 
is a horizontal 
distribution $H\,\subset\, T\EE$ such that, for all $e\,\in\,\EE$,
$$T_{e}\EE\,=\,(V\EE)_e\oplus H_{e}\,,$$
where $V\EE\,:=\, {\rm kernel}(d\pi)$ is the vertical tangent bundle for the projection
$\pi$, and the subbundle $H\,\subset\, T\EE$ is preserved by the action of the group $G$ on $T\EE$
given by the action of $G$ on $\EE$.
This is equivalent to having a $\mathfrak g$-valued $1$-form $A$ on $\EE$
(with values in $\mathfrak{g}$), that is, a smooth section of the bundle $T^*\EE\otimes \lieg$;
again we require this $A$ to be invariant under the action of $G$,
acting by a combination of the given action on $\EE$ and the adjoint action 
on $\lieg$ (equivalently, the map $T\EE\, \longrightarrow\, \lieg$ given by $A$ is
$G$-equivariant). Also, $A$ should restrict to the canonical right-invariant $\mathfrak g$-valued $1$-form
on the fibers of $\pi$. For details, see \cite[Ch.~II]{KN}.

Let now $$\rho\,:\, G\, \longrightarrow\, \GL(\VV)$$ be a
holomorphic representation of $G$ on a complex vector space 
$\VV$, and let $V\,=\,\EE(\VV)$ be the associated smooth complex vector bundle. Recall that a connection
$A$ on $\EE$ defines a {\bf covariant derivative} on $V$, that is a $\mathbb C$--linear map 
$$
d_A\,:\,\Omega^0(X,V)\,\longrightarrow \,\Omega^1(X,V)
$$
that satisfies the Leibniz rule; here $\Omega^p(X,V)$ denotes the space of all smooth sections of 
$V\otimes \bigwedge^p T^*X$.

Now, take $G$ to be a reductive complex affine algebraic group equipped with a 
real structure $\sigma_{G}$, and let $(X,\,\sigma_{X})$ be a Klein surface; also fix
an element $c\,\in\, Z_{2}^{\sigma_{G}}$.
 
Given a smooth principal $G$-bundle $\EE$ on $X$ equipped with a 
$(\sigma_{X},\sigma_{G},c)$-{\bf real structure} $\sigma_{\EE}$, we can impose a
reality condition for a connection on $\EE$. A connection $A$ on $\EE$ is said to be $\sigma_\EE$-{\bf compatible}
if the horizontal distribution $H\, \subset\,T\EE$ corresponding to $A$ satisfies the condition
$$
d(\sigma_{\EE}) (H_{e})\,=\, H_{\sigma_{\EE}(e)}\,,\ \ \text{for all}\ \, e\,\in\, \EE\, .
$$
This is equivalent to the condition that the corresponding
$\mathfrak g$-valued one-form $A\,:\,T\EE\,\longrightarrow\, \lieg$
is $(d\sigma_{\EE},\,d\sigma_{G})$-real, meaning
$$
 d\sigma_{G}\circ A\circ d\sigma_{\EE}\,=\, A\,.
$$
Let $\AAA(\EE)$ be the set of all connections on $\EE$. We shall denote by 
$$\AAA(\EE,\sigma_\EE)\,\subset\, \AAA(\EE)$$ the subset consisting of all
$\sigma_\EE$-compatible connections. 

Let $$K\,\subset\, G$$ be a maximal compact subgroup of $G$ preserved by the involution 
$\sigma_{G}$ (see Proposition \ref{max}). Therefore, the involution $\sigma_{\EE}$ of $\EE$ 
defines an involution of the quotient space $\EE/K$.
A $C^\infty$ reduction of structure group $h\, :\, X\, \longrightarrow\, \EE/K$
of the principal $G$-bundle $\EE$, from $G$ to $K$, is called $\sigma_{\EE}$-\textbf{compatible} if the image of 
$h$ is preserved by this involution of $\EE/K$. Note that $h$ is $\sigma_{\EE}$-compatible if and 
only if the reduction of structure group of $\EE$ to the subgroup $K$ $$\EE_{K}\, \subset\, \EE$$ 
corresponding to $h$ satisfies the condition $\sigma_{\EE}(\EE_K)\,=\,\EE_{K}$.

Let $$\EE_{K}\, \subset\, \EE$$ be a 
$\sigma_{\EE}$-compatible reduction of structure group to the maximal compact subgroup $K$ of 
$G$. We denote by $\sigma_{\EE_{K}}$ the restriction of the involution $\sigma_{\EE}$ to 
$\EE_{K}$. The space of all connections on the principal $K$-bundle $\EE_{K}$ will be denoted by 
$\AAA(\EE_K)$. Let $$\AAA(\EE_{K},\sigma_{\EE_{K}})\,\subset\, \AAA(\EE_K)$$ 
be the subset consisting of all $\sigma_{\EE_{K}}$-compatible connections. Recall that a
connection $\nabla\, \in\, \AAA(\EE_K)$ lies in $\AAA(\EE_{K},\sigma_{\EE_{K}})$ if and only
if the horizontal distribution ${\mathcal H}(A)\, \subset\, T\EE_K$ for $A$ is preserved
by the involution of $T\EE_K$ given by $\sigma_{\EE_{K}}$.

There is a natural bijection between the space of complex structures
$\CCC(\EE)$ on $E$ and the space of connections $\AAA(\EE_K)$ on $\EE_K$. 
This bijection is given by the {\bf Chern map}
\begin{equation}\label{ech}
C\,:\,\CCC(\EE)\,\longrightarrow\,\AAA(\EE_K)
\end{equation}
which is constructed as follows.
For a complex structure $J\, \in\, \CCC(\EE)$, the horizontal 
distribution on $\EE_K$ defined by the connection $C(J)$ is given 
by 
$$
J(T\EE_{K}))\cap T\EE_{K}\, \subset\, T\EE_K\, .
$$ 
For details see \cite[pp.~191--192, Proposition 5]{At1} and 
\cite[p. 586]{singer}.

If now $\EE$ has a $\sigma_\EE$-real structure, and $\EE_K$ is a 
$\sigma_\EE$-compatible reduction of structure group, we have the following.

\begin{Pro}\label{pr:dolbeautyconexiones}
The Chern map defines a bijection between $\CCC(\EE,\sigma_{\EE})$ and
$\AAA(\EE_{K},\sigma_{\EE_{K}})$.
\end{Pro}

\begin{proof}
Let $J$ be a $\sigma_{\EE}$-real complex structure on $\EE$. Let $\widetilde{C(J)}\, \subset\, T\EE_K$
be the horizontal distribution for the connection $C(J)$ on $\EE_K$, where $C$ is the map in \eqref{ech}.
So $\widetilde{C(J)}$ is the unique subbundle of $T\EE_K$ satisfying the following three conditions:
\begin{enumerate}
\item $J(\widetilde{C(J)})\,=\, \widetilde{C(J)}$, and

\item $\widetilde{C(J)}\oplus \text{kernel}(d\pi')\,=\, T\EE_K$, where $\pi'\, :\,
\EE_K\, \longrightarrow\, X$ is the natural projection, and

\item $\widetilde{C(J)}$ is preserved by the action of $K$ on $\EE_K$.
\end{enumerate}

For any $e\, \in\, \EE_K$,
if $Y\,\in\, \widetilde{C(J)}_{e}$, then $Y\,=\,J(Z),$ for some $Z\,\in\, \widetilde{C(J)}_{e}$. 
By \eqref{pr:cond1} we have
$$d\sigma_{\EE}(Y)\,=\,d\sigma_{\EE}(J(Z))\,=\,J (d\sigma_{\EE}(-Z))\,\in
\, J (T_{\sigma_{\EE}(e)}(\EE_{K}))\, ,$$
but we also have
$d\sigma_{\EE}(Y)\,\in\, T_{\sigma_{\EE}(e)}(\EE_{K})$. So $d\sigma_{\EE}(\widetilde{C(J)}_{e})$
is also preserved by $J(\sigma_{\EE}(e))\, \in\, \text{GL}(T_{\sigma_{\EE}(e)}\EE)$. Consequently, from
the above uniqueness property of $\widetilde{C(J)}$ it follows that
$$
d\sigma_{\EE}(\widetilde{C(J)})\,=\, \widetilde{C(J)}\, ;
$$
recall that $J$ is preserved by the action of $G$ on $\EE$.
Therefore, we conclude that $C(J)\, \in\, \AAA(\EE_{K},\sigma_{\EE_{K}})$.

The converse is proved similarly.
\end{proof}

Let $X$ be a Riemann surface with complex structure $J_{X}$ and equipped with a real structure 
$\sigma_{X}$. Let $G$ be a reductive complex affine algebraic group equipped
with real structure $\sigma_{G}$, and let $\VV$ be a complex vector space equipped 
with real structure $\sigma_{\VV}$. Let 
$\rho\,:\,G\,\longrightarrow \,\GL(\VV)$ be a $(\sigma_{G},\,\sigma_{\VV})$-compatible holomorphic
representation. Let $\EE$ be a $C^\infty$ principal $G$-bundle over $X$ equipped with a 
$(\sigma_{X},\sigma_{G},c)$-real structure $\sigma_{\EE}$ as well as a complex structure $J$. The 
associated vector bundle $V\,=\,\EE(\VV)$ is equipped with the following structures:
\begin{itemize}
\item an involution 
$\sigma_{V}$, induced by $\sigma_{\EE}$ and $\sigma_{\VV}$, and

\item a complex structure $J_V$ whose Dolbeault operator is given by 
$$\overline{\partial}_{J_{V}}\varphi\,:=\, \frac{d\varphi+J_{V}\circ d\varphi \circ J_{X}}{2}$$ for all 
$\varphi\,\in\, \Omega^{0}(X,V)$ (considered as a map $\varphi\, :\, X\, \longrightarrow\, V$
between complex manifolds).
\end{itemize}
If the complex structure $J$ on $\EE$
is $\sigma_{\EE}$-real, then $\overline{\partial}_{J_{V}}$ anti-commutes with $J_{V}$, which is equivalent
to the condition that $(V,\, \sigma_{V})$ is a real holomorphic vector bundle;
see Proposition \ref{equivalenciaprincipal112}.
If $\EE$ is equipped with a $\sigma_{\EE}$-compatible reduction of structure
group $\EE_K\, \subset\, \EE$ to the maximal
compact subgroup $K$, then
$A\,:=\,C(J)$ is a $\sigma_{\EE_{K}}$-compatible connection, defining 
a covariant derivative $d_A$ on $V$ which is related to the Dolbeault operator by
$$\overline{\partial}_{J_{V}}\varphi\,=\,\overline{\partial}_{A}\varphi\,:=\,\pi^{0,1} d_A\varphi$$
for $\varphi\,\in\, \Omega^{0}(X,V)$, where $\pi^{0,1}$ is the projection of complex $1$-forms to
$(0,\, 1)$-forms.

\begin{Exam}[{Compatible Hermitian metric}]
Let $\EE$ be a $C^{\infty}$ principal $\GL(n,\CC)$-bundle, and let $V\,:=\,\EE(\CC^{n})$
be the associated complex vector bundle of rank $n$ on $X$ for the standard representation
of $\GL(n,\CC)$ on $\CC^{n}$. Let $\sigma_{\GL(n,\CC)}$ be the real structure
on $\GL(n,\CC)$ that sends any matrix $A$ to its conjugate matrix $\overline{A}$. Then a 
$(\sigma_{X},\sigma_{GL(n,\CC)},c)$-real structure $\sigma_{\EE}$ on $\EE$ corresponds to a
$\sigma_{X}$-real structure $\sigma_{V}$ on $V$.
Giving a $\sigma_{\EE}$-compatible reduction on $\EE$ to ${\rm U}(n)$ is equivalent to giving an Hermitian structure
$h\,:\,V\,\longrightarrow\, \overline{V}^{*}$ on $V$ such that the diagram
\begin{equation}\label{eq:metricahermiticadual}
\xymatrix{
V\ar[d]^{\sigma_{V}}\ar[r]^{h}&\overline{V}^{*}\ar[d]^{\overline{\sigma}_{V^{*}}}\\
V\ar[r]^{h}&\overline{V}^{*}}
\end{equation}
commutes, where $\overline{\sigma}_{V^{*}}$ is the transpose of the conjugate of $\sigma_{V}$.
This condition is equivalent to the condition that fiberwise $\sigma_V$ is an isometry.

In this situation, the Chern correspondence defines a bijection 
between the set of $\sigma_V$-compatible Dolbeault operators on $V$ and 
the set of all $\sigma_V$-compatible unitary connections on $(V,\,h)$.
\end{Exam}

\section{Real twisted Higgs pairs, stability and equations}\label{Section3}

\subsection{Real twisted Higgs pairs}\label{Subsection3.1}

Let $X$ be a compact connected Riemann surface, and let $G$ be a connected reductive complex 
affine algebraic group. Let $\VV$ be a complex vector space and $\rho\,:\, G\,\longrightarrow 
\,\GL(\mathbb{V})$ a holomorphic representation. Let $E$ be a holomorphic principal $G$-bundle
on $X$. We denote by $V$ the holomorphic vector bundle 
$E(\mathbb{V})$ associated to $E$ via $\rho$. Let $L$ be a holomorphic line bundle over $X$.

A $L$-\textbf{twisted Higgs pair of type $\rho$} is a pair $(E,\,\varphi)$ consisting of a 
holomorphic principal $G$-bundle $E$ over $X$ and a holomorphic section $\varphi$ of the 
holomorphic vector bundle $V\otimes L$, where $V$ is defined above.

Let $\sigma_X$ and $\sigma_G$ be real structures on $X$ and $G$, respectively. Let $\sigma_{\VV}$ be 
a real structure on $\VV$ and $\rho\,:\, G\,\longrightarrow \,\GL(\mathbb{V})$ a 
$(\sigma_{G},\sigma_{\VV})$-compatible holomorphic representation. This means that
\begin{equation}\label{ez1}
\rho(g)(\sigma_{\VV}(v)) \,=\, (\rho(\sigma_G (g)))(v)
\end{equation}
for all $g\, \in\, G$ and $v\, \in\, \VV$. Let $\sigma_E$ be a real structure on $E$.

Let $Z$ be the center of $G$, and let $$Z_{2}^{\sigma_{G}}\, \subset\, Z$$
be the subgroup consisting of all $t\, \in\, Z,$ such that
\begin{itemize}
\item the order of $t$ is two (if $t$ is not the identity element), and

\item $\sigma_{G}(t)\,=\, t$.
\end{itemize}
Take any $c\,\in\, 
Z_{2}^{\sigma_{G}}\bigcap\ker{\rho}$. Let $(E,\,\sigma_{E})$ be a 
$(\sigma_{X},\sigma_{G},c)$-real principal $G$-bundle over $(X,\,\sigma_X)$.

Consider the map $E\times \VV\, \longrightarrow\, E\times\VV$ that sends any $(e,\, v)$ to
$(\sigma_E(e),\,\sigma_\mathbb{V}(v))$. This map descends to
an anti-holomorphic map $$\sigma_V\,:\,V\,\longrightarrow\, V\, ,$$
because $\rho$ satisfies \eqref{ez1}. This map $\sigma_V$
defines a real structure on the vector bundle $V$, because $c\,\in \,\Ker(\rho)$. 
Let $(L,\,\sigma_{L})$ be a real holomorphic
line bundle over $(X,\,\sigma_X)$. The tensor 
product $\sigma_{V}\otimes\sigma_{L}\, :\, V\otimes L \,\longrightarrow\, V\otimes L$
is clearly a real structure on $V\otimes L$. 

A $(\sigma_{X}, \sigma_{G},c,\sigma_{L},\sigma_{\VV},\pm)$-\textbf{real structure} on a 
$L$-twisted Higgs pair $(E,\,\varphi)$ of type $\rho$ is an anti-holomorphic map $\sigma_E\,:\, 
E\,\longrightarrow\, E$ such that $(E,\,\sigma_E)$ is a $(\sigma_{X},\sigma_{G},c)$-real 
holomorphic principal $G$-bundle over $(X,\sigma_X)$, and the section $\varphi\, \in\, H^0(X,\,
V\otimes L)$ satisfies the equation
\begin{equation}\label{eq:realhiggs}
\sigma_{V}\otimes\sigma_{L}(\varphi)\,=\,\pm\varphi\circ\sigma_{X}\,=:\, \pm \sigma^*_X\varphi\, .
\end{equation}
Denote by $\pmb{\sigma}\,=\,(\sigma_{X}, \sigma_{G},c,\sigma_{L}, \sigma_{\VV},\pm)$. The triple 
$(E,\varphi,\sigma_E)$ will sometimes be referred to as a $\ssigma$-\textbf{real twisted Higgs 
pair of type $\rho$}.

We give some concrete examples of real Higgs pairs.
 
\begin{Exam}\label{ex:realHiggsbundles}
Let $(G,\,\sigma_{G})$ be a reductive complex affine algebraic group equipped with an anti-holomorphic 
involution. Let $(L,\,\sigma_{L})\,=\,(K_{X},\,\sigma_{K_{X}})$ be the canonical line
bundle of $X$ equipped with 
the anti-holomorphic involution induced by $\sigma_{X}.$ Set $(\VV,\,\sigma_{\VV})\,=\,(\gG,\, 
d\sigma_{G})$. The adjoint representation $\Ad\,:\,G\,\longrightarrow \,\GL(\gG)$ is 
$(\sigma_G,\,\sigma_{\VV})$-real. In this case, $\ssigma$-real $K_{X}$-twisted Higgs pairs of 
type $\Ad$ are known in the literature as \textbf{pseudo-real $G$-Higgs bundles} (see 
\cite{BGH:98}, \cite{BGH:99} and \cite{BH:11}).
\end{Exam}

\begin{Exam} Let $G\,=\,\GL(n,\CC),$ $\VV\,=\,\Sym^{2}(\CC^{n})$ and $\rho$ the 
representation on $\VV$ induced by the standard 
representation of $\GL(n,\CC)$ on $\CC^n$. Let $L$ be a holomorphic line bundle on $X$.
A $L$-twisted Higgs pairs of type 
$\rho$ is an $L$-quadratic pair. They are also known in the literature as conic bundles or 
quadric bundles (see, for example, \cite{GGM2,Oliveira}). If we equip $G$ and $\VV$ with complex 
conjugations $\sigma_{G}(A)\,=\,\overline{A}$ and $\sigma_{\VV}(v)\,=\,\overline{v}$, for
all $A\,\in\,\GL(n,\CC)$ and 
$v\,\in\, \VV$, then $\rho$ is a $(\sigma_{G},\,\sigma_{\VV})$-compatible representation, thus
defining a real structure $\sigma_V$ on $V$. Let $\sigma_{L}$ 
be a real structure on $L$. Then $\ssigma$-real $L$-twisted Higgs pairs of type $\rho$ 
are \textbf{real $L$-quadratic pairs} $(V,\,\varphi,\,\sigma_{V})$; by this we mean that 
$(V,\,\sigma_{V})$ is a real holomorphic vector bundle of rank $n$, and $\varphi\,:\,V
\,\longrightarrow\, V^{*}\otimes L$ 
is a real section, meaning that $\varphi\circ \sigma_{V}\,=\,\sigma_{V}^{t}\otimes \sigma_{L}\circ 
\varphi.$
\end{Exam}

\begin{Exam}
Let $G\,=\,\GL(n_1,\CC)\times \GL(n_2,\CC),~$ $\VV\,=\,\Hom(\CC^{n_2},\,\CC^{n_1})$, $\rho$ 
is the representation of $G$ given by $\rho(A,B)(T)\,=\,A\circ T\circ B^{-1},$ and $L$ is the 
trivial holomorphic line bundle $\Oo_X$. The
$L$-twisted Higgs pairs of type $\rho$ are known in the literature as holomorphic triples 
(see, for example, \cite{bradlowtriples}).
We equip $G$ and $\VV$ with real structures $\sigma_{G}$, $\sigma_{\VV}$,
defined by $\sigma_{G}(A\times B)\,=\,\overline{A}\times \overline{B}$ and 
$\sigma_{\VV}(T)\overline{v}\,=\,\overline{Tv}$, for all $A\,\in\,\GL(n_1,\CC),$ $B\,\in\,\GL(n_2,\CC),$ $T\,\in\, 
\Hom(\CC^{n_2},\,\CC^{n_1})$ and $v\,\in\,
\CC^{n_2}$. Then $\rho$ is a $(\sigma_{G},\,\sigma_{\VV})$-compatible representation. We 
equip $\Oo_X$ with the real structure induced by
a real structure $\sigma_{X}$ on $X$. The $\ssigma$-real $\Oo_X$-twisted Higgs 
pairs of type $\rho$ are \textbf{real holomorphic triples},
that consist of two real holomorphic vector bundles $(V_{1},\,\sigma_{V_{1}})$, 
$(V_{2},\,\sigma_{V_{2}})$ and a real morphism 
$\phi\,:\,(V_{1},\,\sigma_{V_{1}})\,\longrightarrow\,(V_{2},\,\sigma_{V_{2}}).$
\end{Exam}

\subsection{Stability}\label{Subsection3.3}

Let $K$ be a $\sigma_G$-invariant maximal compact subgroup of a 
connected reductive complex affine algebraic group $G$,
and let
$$
B\, \in\, \text{Sym}^2({\mathfrak g}^*)^G
$$
be a $K$-invariant non-degenerate bilinear form on the Lie algebra ${\mathfrak g}\,=\,\Lie(G)$ which
is positive on ${\mathfrak k}\,=\,\Lie(K).$ 
Let $s$ be an element of $\sqrt{-1}\cdot\mathfrak{k}$; we note that the following objects are associated to $s$:
\begin{itemize} 
\item $P_{s}\,:=\,\{g\,\in\, G\, \mid\, e^{ts}ge^{-ts}~\text{ is bounded when }~t\rightarrow\infty\}$
is a parabolic subgroup of $G$,

\item $L_{s}\,:=\,\{g\,\in\, P_{s}\,\mid\, \Ad(g)s\,=\,s\}$ is a Levi subgroup of the
parabolic subgroup $P_s$,
where $\Ad\,:\,G\,\longrightarrow \,\Aut(\gG)$ is the Adjoint representation,

\item $\mathfrak{p}_{s}\,:=\,\{x\,\in\, \mathfrak{g}\,\mid\, \Ad(e^{ts})x~\text{ is bounded when }~
t\rightarrow\infty\}$ is the Lie algebra of $P_s$,

\item $\mathfrak{l}_{s}\,:=\,\{x\,\in\, \mathfrak{g}\,\mid\, \left[x,\,s\right]\,=\,0\}$ is the
Lie algebra of $L_s$,

\item $\mathbb{V}_{s}\,:=\,\{v\,\in\, \mathbb{V}\,\mid\,\rho(e^{ts}v)~\text{ is bounded when }
~t\rightarrow\infty\}$,

\item $\mathbb{V}_{s}^{0}\,:=\,\{v\,\in\, \mathbb{V}\,\mid\, \rho(e^{ts}v)\,=\,v~ \text{ for all }~t\}$, and

\item $\chi_{s}$ is the dual of $s$ with respect to $B$, meaning that $B$ induces an isomorphism 
$B\,:\,\kK\,\longrightarrow\, \kK^{*}$, using which we have 
$\chi_{s}\,:=\,B(s)|_{\pP_s\cap \kK}$.
\end{itemize}

\begin{Rem}\label{rem:defis_chi}
A character of $\pP_{s}$ is a complex linear map $\pP_{s}\,\longrightarrow \,\CC$ that factors
through the quotient $\pP_{s}/[\pP_{s},\,\pP_{s}].$
Since $B(s,\,[\pP_s,\,\pP_s])\,=\,0$, the above defined homomorphism $\chi_{s}$ produces
a strictly anti-dominant character of $\pP_s.$ Conversely, given an anti-dominant character
$\chi$ of $\pP_s$, we have $\chi\,\in\, (\pP_{s}/[\pP_s,\,\pP_s])^{*}\,\cong\,(\zZ_{L_{s}})^{*},$
where $\zZ_{L_{s}}$ is the center of the Levi subalgebra $\lL_s.$ For
$s_{\chi}\,=\,B^{-1}(\chi)$, we have $s_{\chi}\,\in\, \zZ_{L_{s}}\,\subset \,\sqrt{-1}\cdot\kK.$
\end{Rem}

We now recall the definition of the degree of a
principal $G$-bundle $E$ with respect to an element $s\,\in\, \sqrt{-1}\cdot\kK$ and
a holomorphic reduction $\sigma$ of the structure group of $E$ from $G$ to $P_s$.
Let $E_{P_s}\, \subset\, E$ be the holomorphic principal $P_s$-bundle corresponding to $\sigma$. 
By \cite[Lemma 2.4]{GGM:21}, there is a rational number $q$ such
that $(\chi_{s})^{q}$ lifts to a character $\widehat{\chi}_s$ of $P_s$. Then 
$E_{P_{s}}(\widehat{\chi}_s)$ is a line bundle. Now define the degree
\begin{equation}
\label{eq:deg}
\deg E(\sigma,\, s)\,:=\,\frac{1}{q}\deg E_{P_{s}}(\widehat{\chi}_s)\,.
\end{equation}

Let $\sigma_{G}$ be a conjugation on $G$, and let $(E,\,\sigma_{E})$ be a $(\sigma_{X}, 
\sigma_{G},c)$-real holomorphic principal $G$-bundle over $X$. We denote by $\Ad(E)$ the group scheme 
$E\times_{\Ad}G$ over $X$, which is associated to $E$ via the adjoint action of $G$
on itself. The real structures $\sigma_{E}$ and $\sigma_{G}$ together define a 
real structure $\sigma_{\Ad(E)}$ on $\Ad(E)$, because the homomorphism $\Ad\,:\, G\,\longrightarrow\, \Aut(G)$
given by the adjoint action is compatible with $\sigma_{G}$ (see Example \ref{adjointrep}).

The main change in the
definition of stability for $\ssigma$-real $L$-twisted Higgs pairs
vis-\`a-vis the definition of stability of the underlying Higgs pairs
is that we must consider only holomorphic
reductions $E_{P_s} \,\subset\, E$ of $E$ from $G$ to $P_s$ such that 
\begin{equation}\label{eq:stabilitycondition}
\sigma_{\Ad(E)}(\Ad(E_{P_{s}}))\,=\,\Ad(E_{P_{s}})\, .
\end{equation}
More specifically, for $\alpha\,\in\, \mathfrak{z}(\kK)$, a $\ssigma$-real $L$-twisted Higgs
pair $(E,\varphi,\, \sigma_E)$ is
\begin{itemize}
\item $\alpha$-\textbf{semistable} if for every $s\,\in\, \sqrt{-1}\kK$ and for every holomorphic
reduction of structure group $\sigma$ of $E$ to $P_{s}$ such that
\begin{enumerate}
\item the holomorphic principal $P_s$-bundle $E_{P_s}\, \subset\, E$ corresponding to $\sigma$ satisfies
the equation in \eqref{eq:stabilitycondition}, and

\item $\varphi\,\in\, H^{0}(X,\,E_{P_s}(\mathbb{V}_{s})\otimes L)$,
\end{enumerate}
the following inequality holds:
$$\deg E(s,\sigma)-B(s,\alpha)\,\geq\, 0\, ;$$ 

\item $\alpha$-\textbf{stable} if for every $s\,\in\, \sqrt{-1}\kK\setminus \Ker(d\rho)$ 
and for every reduction $\sigma$ of structure group of $E$ to $P_{s}$ satisfying 
(\ref{eq:stabilitycondition}) as well as the condition
$\varphi\,\in\, H^{0}(X,\,E_{P_s}(\mathbb{V}_{s})\otimes L)$, the inequality
$$\deg E(s,\sigma)-B(s,\alpha)\,>\,0$$ 
holds;

\item $\alpha$-\textbf{polystable} if it is $\alpha$-semistable and furthermore, for every 
reduction $\sigma$ of structure group of $E$ to a
parabolic subgroup $P_{s}$ satisfying (\ref{eq:stabilitycondition}) 
as well as the two conditions $\varphi\,\in\, H^{0}(X,\, E_{P_s}(\mathbb{V}_{s})\otimes L)$ and
$$\deg E(s,\,\sigma)-B(s,\,\alpha)\,=\,0\,,$$ there is a holomorphic reduction of
structure group
\begin{equation}\label{els}
E_{L_s}\, \subset\, E_{P_s}
\end{equation}
of $E_{P_s}$ to the Levi subgroup $L_{s}\, \subset\, P_s$ such that
$$\sigma_E(E_{L_s})\,=\, E_{L_s}$$ and $\varphi\,\in\, 
H^{0}(X,\, E_{L_s}(\mathbb{V}^{0}_{s})\otimes L)$.
\end{itemize}

\begin{Rem} If the condition in (\ref{eq:stabilitycondition}) is dropped, then we obtain the
usual definition of stability for the $L$-twisted Higgs pair $(E,\, \varphi)$ underlying the triple
$(E,\,\varphi,\,\sigma_{E})$. The naive condition for $\ssigma$-real $L$-twisted Higgs pairs,
considering only those parabolic subgroups $P_{s}$ such that $\sigma_{G}(P_s)\,=\,P_s\,$, is not
the right one for stability, since there are some cases (such as the compact real form of $G$)
for which there are no $\sigma_{G}$-invariant parabolic
subgroups and, therefore, such a stability condition would be trivially satisfied.
\end{Rem}
 
\begin{Rem}\label{rem:weaker}
Let $(E, \,\varphi,\, \sigma_E)$ be a $\ssigma$-real $L$-twisted 
Higgs pair. If the underlying Higgs pair $(E,\, \varphi)$ is $\alpha$-stable
(respectively, $\alpha$-semistable), then $(E,\,\varphi,\, \sigma_E)$ is 
obviously $\alpha$-stable (respectively, $\alpha$-semistable).
\end{Rem}

\subsection{Hermite--Einstein--Higgs equation}\label{sec:notation}\label{Subsection3.2}

Let $G$ be a connected reductive complex affine algebraic group equipped with a real structure 
$\sigma_{G}\,$. Let $K\,\subset\, G$ be a maximal compact Lie subgroup preserved by the 
involution $\sigma_G\,$ (see Proposition \ref{max}). Fix a non-degenerate $G$-invariant bilinear form
$$
B\, \in\, \text{Sym}^2({\mathfrak g}^*)^G
$$
which is
\begin{itemize}
\item positive on ${\mathfrak k}\,=\,\Lie(K)$, and

\item compatible with $d\sigma_{G}$, meaning that $(d\sigma_{G})^{*}B\,=\, B$.
\end{itemize}
Let $(X,\,\sigma_{X})$ be a compact Klein surface, and let $\omega$ be a $\sigma_{X}$-real K\"ahler form 
on $X$. Let $(L,\,\sigma_{L})$ be a real holomorphic line bundle over $(X,\,\sigma_{X})$. Fix a 
$\sigma_{L}$--compatible Hermitian metric $h_{L}$ on $L$. We denote by $F_L$ the 
curvature of the corresponding Chern connection on $L$. Let $\VV$ be a complex vector space 
equipped with a real structure $\sigma_{\VV}$, and let
\begin{equation}\label{er}
\rho\,:\, G\,\longrightarrow\, \GL(\mathbb{V})
\end{equation}
be a $(\sigma_{G},\sigma_{\VV})$-real holomorphic representation. Let $h_\VV$ be a $K$-invariant Hermitian
inner product on the vector space $\VV$ compatible with $\sigma_{\VV}$, 
meaning $h_{\VV}(\sigma_{\VV}(v),\,\sigma_{\VV}(v'))=\,h_{\VV}(v,\,v')$ for 
all $v,\, v'\,\in\, \VV\,.$

As above, take a holomorphic principal $G$-bundle $E$ on $X$ equipped with a real structure 
$\sigma_{E}$. Let $h$ be a $\sigma_{E}$--compatible reduction of the structure group of $E$ from 
$G$ to the subgroup $K$ (see Section \ref{subsection2.3}). The resulting real $K$-bundle 
will be denoted by $(E_K,\sigma_{E_K})$. Let $$A_{h}\,\in\, \AAA(E_{K},\sigma_{E_{K}})$$
be the unique $\sigma_{E_{K}}$-connection on $E_{K}$ compatible with 
the holomorphic structure of $E$ (see Proposition \ref{pr:dolbeautyconexiones}). Let $F_{h}\,\in\, \Omega^{1,1}((E_{K}(\kK))$ be the
curvature of $A_{h}$. 

The associated holomorphic vector bundle $V\,:=\,E(\mathbb{V})$ has a Hermitian structure $h_{V}$ 
which is constructed using the reduction $h$ and the Hermitian structure $h_{\VV}$ on $\VV$. Since $h_{\VV}$ is $K$-invariant and compatible with 
$\sigma_{\VV},$ this Hermitian metric $h_{V}$ is in fact $\sigma_V$--compatible,
where $\sigma_V$ is the real structure on $V$ defined by $\sigma_G$ and $\sigma_\VV$. Note that the 
Hermitian structure $$h_{V\otimes L}\,:=\,h_{V}\otimes h_{L}$$ on $V\otimes L$ is 
compatible with the involution $\sigma_{V}\otimes\sigma_L$. The 
representation $\rho$ in \eqref{er} induces a linear map
$$
\xymatrix{
\uU(\mathbb{V})^{*}\ar@/^2pc/[rr]^{Q}\ar[r]^{(d\rho)^{*}} &\mathfrak{k}^{*}\ar[r]^{B} &\mathfrak{k}\,,}
$$
where $\uU(\mathbb{V})\, \subset\, \text{End}({\mathbb V})$ is the
subalgebra consisting of the skew-Hermitian endomorphisms. This homomorphism
$Q$ produces a homomorphism of vector bundles associated to the principal $K$-bundle $E_K$
$$\widetilde{Q}\,:\, E_K(\uU(\VV))^{*}
\,\longrightarrow\, E_K(\mathfrak{k})\,.$$
The pairing $$\text{trace}\, :\, E_K(\uU(\VV))\otimes E_K(\uU(\VV))\, 
\longrightarrow\, X\times\mathbb C\, ,$$
being nondegenerate, identifies $E_K(\uU(\VV))$ with the dual vector bundle $E_K(\uU(\VV))^{*}$. Using this
identification, the above homomorphism $\widetilde{Q}$ would be considered as a homomorphism
\begin{equation}\label{qp}
\widetilde{Q}\,:\, E_K(\uU(\VV)) \,\longrightarrow\, E_K(\mathfrak{k})\,.
\end{equation}

Now take any $L$-twisted Higgs field
$$
\varphi\, \in\, H^0(X,\, V\otimes L)
$$
on $E$. It produces a $C^\infty$ section
$$
\widehat{\varphi} \, \in\, C^\infty(X,\, E_K(\uU(\VV)))
$$
as follows: Consider the $C^\infty$ section
$$
\overline{\varphi}\, \in\, C^\infty(X,\, \overline{V\otimes L})\,=\,
C^\infty(X,\, \overline{V}\otimes\overline{L})\, .
$$
Note that
$$\varphi\otimes \overline{\varphi}\, \in\,
C^\infty(X,\, V\otimes L)\otimes
C^\infty(X,\, \overline{V}\otimes\overline{L})
$$
produces a $C^\infty$ section
\begin{equation}\label{evp}
\varphi'\, \in\, C^\infty(X,\, (V\otimes\overline{V})\otimes (L\otimes\overline{L}))\, .
\end{equation}
Now, the Hermitian structure $h_L$ on $L$ identifies $\overline{L}$ with $L^*$ (it
is a $C^\infty$ isomorphism), and
the Hermitian structure $h_V$ on $V$ identifies $\overline{V}$ with $V^*$. Therefore,
using the natural pairing $L\otimes L^*\, \longrightarrow\, {\mathcal O}_X$, the
$C^\infty$ section $\varphi'$ in \eqref{evp} produces a $C^\infty$ section
\begin{equation}\label{vpp}
\varphi''\, \in\, C^\infty(X,\, V\otimes V^*)\, .
\end{equation}
It is straightforward to check that the section $\varphi''$ is
pointwise Hermitian (same as self-adjoint), meaning $\varphi''(x)\, \in\, \text{End}(V_x)$
is Hermitian for all $x\, \in\, X$.

To see that the section $\varphi''$ is pointwise Hermitian, let $\mathbb W$ be a finite dimensional
complex vector space equipped with a Hermitian structure $h_{\mathbb W}$. Take any $v\, \in\, \mathbb W$.
Note that $v$ produces an endomorphism
$$
F_v\, \in\, \text{End}({\mathbb W})\, =\, {\mathbb W}\otimes{\mathbb W}^*
$$
defined by
$$
F_v(z)\,:=\, h_{\mathbb W}(z,\, v)\cdot v
$$
for all $z\, \in\,{\mathbb W}$. Now for all $z,\, w\, \in\,{\mathbb W}$, we have
$$
h_{\mathbb W}(F_v(z),\, w)\,=\, h_{\mathbb W}(h_{\mathbb W}(z,\, v)\cdot v,\, w)\,=\, h_{\mathbb W}(z,\, v)\cdot
h_{\mathbb W}(v,\, w)
$$
$$
=\, \overline{h_{\mathbb W}(w,\, v)}\cdot h_{\mathbb W}(z,\, v)
\,=\, h_{\mathbb W}(z,\, h_{\mathbb W}(w,\, v)\cdot v)\,=\, h_{\mathbb W}(z,\, F_v(w))\, .
$$
Consequently, $F_v\, \in\, \text{End}({\mathbb W})$ is Hermitian. This immediately implies that
the section $\varphi''$ in \eqref{vpp} is pointwise Hermitian, and hence we conclude that
\begin{equation}\label{vpp2}
\widehat{\varphi} \, :=\, \frac{\sqrt{-1}}{2}\varphi'' \, \in\, C^\infty(X,\, E_K(\uU(\VV)))\, .
\end{equation}

To describe $\widehat{\varphi}$ in \eqref{vpp2} more concretely, if $\varphi$
is locally of the form $\varphi_V\otimes\varphi_L$, where $\varphi_V$ and $\varphi_L$ are
locally defined $C^\infty$ sections of $V$ and $L$ respectively, then
$$
\widehat{\varphi}(\mathbf{s})\,:=\, \frac{\sqrt{-1}}{2}h_V(\mathbf{s},\, \varphi_V)
h_L(\varphi_L,\, \varphi_L)\cdot \varphi_V\, ,
$$
where $\mathbf{s}$ is any locally defined $C^\infty$ section of $V$. It is straight-forward
to check that the above expression of $\widehat{\varphi}$ does not depend on the choice
of the local decomposition $\varphi\,=\, \varphi_V\otimes\varphi_L$.

Note that
\begin{equation}\label{z1}
\widehat{c\varphi}\,=\, |c|^2\widehat{\varphi}
\end{equation}
for all $c\, \in\, \mathbb C$.

Now define
\begin{equation}\label{def:muh}
\mu_h(\varphi)\,:=\,\widetilde{Q}(-\widehat{\varphi})\, \in\, C^\infty(X,\,E_K(\mathfrak{k}))\, ,
\end{equation}
where $\widetilde{Q}$ is the homomorphism in \eqref{qp}. From \eqref{z1} it follows
that
$$
\mu_h(c\varphi)\,:=\, |c|^2 \mu_h(\varphi)
$$
for all $c\, \in\, \mathbb C$.

Let $(E,\,\varphi,\,\sigma_{E})$ be a $\ssigma$-real $L$-twisted Higgs pair of type $\rho$.
The center of the Lie algebra $\kK$ is denoted by $\mathfrak{z}(\kK)$.
Take a central element $\alpha\,\in\, \mathfrak{z}(\kK)$ such that $d\sigma_{G}(\alpha)\,=\,-\alpha$. 
A $\sigma_{E}$-compatible reduction $h$ of the structure group of $E$ from $G$ to $K$ is
called \textbf{Hermite--Einstein--Higgs} if it satisfies the equation
 \begin{equation}\label{eq:Einstein}
 \Lambda\,F_h+\mu_h(\varphi)\,=\,-\sqrt{-1}\,\alpha\,,
 \end{equation} 
where $\Lambda$ denotes the contraction of differential forms on $X$ with $\omega$, and $\mu_h$ is the function
in \eqref{def:muh}. Note that we have $\Lambda (F_h)\,\in\, \Omega^{0}(E_{K}(\kK))$, 
since $F_{h}\,\in\,\Omega^{1,1}(E_{K}(\kK))$.

\section{Hitchin--Kobayashi correspondence for real Higgs pairs}\label{sec4n}

The main result of the section, that we prove in the subsequent subsections, is the following.

\begin{The}[{Hitchin--Kobayashi correspondence}]\label{th:Hitchin--Kobayashi}
A $\ssigma$-real $L$-twisted Higgs pair 
$(E,\,\varphi, \,\sigma_E)$ is $\alpha$-polystable if and only
if $(E,\,\varphi, \,\sigma_E)$ admits a $\sigma_{E}$-compatible Hermite--Einstein--Higgs reduction.
\end{The}

\begin{Cor}\label{Cor:realpairspairs}
A $\ssigma$-real $L$-twisted Higgs pair $(E,\,\varphi,\, \sigma_{E})$ is $\alpha$-polystable
if and only if the underlying Higgs pair $(E,\,\varphi)$ is $\alpha$-polystable.
\end{Cor}

\begin{proof}[{Proof of Corollary \ref{Cor:realpairspairs}}]
If a $L$-twisted Higgs pair $(E,\,\varphi)$ is $\alpha$-polystable then $(E,\,\varphi,\,\sigma_{E})$
is $\alpha$-polystable for any $\ssigma$-real structure $\sigma_{E}$ (see Remark \ref{rem:weaker}). 

For the opposite direction, in view of Theorem \ref{th:Hitchin--Kobayashi}, it follows that the
$\alpha$-polystability of a $\ssigma$-real $L$-twisted Higgs pair $(E,\varphi,\sigma_{E})$ is 
equivalent to the existence of a $\sigma_{E}$-compatible solution of the Hermite--Einstein--Higgs 
equation (\ref{eq:Einstein}). The existence of a solution of the Hermite--Einstein--Higgs equation
implies the polystability of the underlying $L$-twisted Higgs pair, by the Hitchin--Kobayashi correspondence 
proved in \cite[Theorem 2.24]{GGM:21}.
\end{proof}

\subsection{Jordan--H\"older reduction}\label{Subsection3.5}

Let $G$ be a connected reductive complex affine algebraic group group, $\VV$ a complex vector 
space and
\begin{equation}\label{hrho}
\rho\,:\,G\,\longrightarrow\, \GL(\VV)
\end{equation}
a holomorphic representation. Let $G'\,\subset\,
G$ be a complex algebraic subgroup. Let $\VV'$ be a complex
linear subspace of $\VV$ such that the action of $G'$ on $\VV$,
obtained by restricting the action of $G$ on $\VV$, preserves this
subspace $\VV'$. Let $$\rho'\,:\,G'\,\longrightarrow\, \GL(\VV')$$
be the restriction of $\rho$.

Take any holomorphic reduction of structure group $E'\, \subset\, E$ to the subgroup $G'$.
Note that the two holomorphic vector bundles $E({\mathbb V})$ and $E'({\mathbb V})$ are
canonically identified. Therefore, $E'({\mathbb V}')$ is a holomorphic subbundle of
$E({\mathbb V})$.

Let $(E,\,\varphi)$ be a $L$-twisted Higgs pair of type $\rho$. A \textbf{reduction of structure
group of $(E,\,\varphi)$, from $(G,\,\rho)$ to $(G',\,\rho')$} is a 
$L$-twisted Higgs pair $(E',\,\varphi')$ of type $\rho'$, where
\begin{itemize}
\item $E'$ is a reduction
of structure group of $E$, from $G$ to $G'$, and

\item $\varphi'\,\in\, H^{0}(X,\, E'(\VV')\otimes L)$ is sent to $\varphi$
by the homomorphism 
$$
H^{0}(X,\, E'(\VV')\otimes L)\,\longrightarrow\,
H^{0}(X,\, E(\VV)\otimes L)
$$ 
induced by the above mentioned vector bundle
injection $\xymatrix{E'(\VV')\ar@{^{(}->}[r] & E(\VV)}$.
\end{itemize}

Let $\sigma_{E}$ be a $(\sigma_{X},\sigma_{G},c,\sigma_{L},\sigma_{\VV},\pm)$-real structure 
on a $L$-twisted Higgs pair $(E,\,\varphi)$. A reduction of structure group $(E',\,\varphi')$ of 
$(E,\,\varphi)$, from $(G,\,\rho)$ to $(G',\,\rho')$, is \textbf{$\sigma_{E}$-compatible} if 
the restriction $\sigma_{E}|_{E'}$ is a $(\sigma_{X},\sigma_{G}|_{G'},c,\sigma_{L}, 
\sigma_{\VV}|_{\VV'},\pm)$-real structure on $(E',\,\varphi')$.

The main result of this subsection is the following:

\begin{The}[Existence of a Jordan--H\"older reduction]\label{th:JHreduction}
Given a $\alpha$-polystable $\ssigma$-real $L$-twisted Higgs pair $(E,\,\varphi,\,
\sigma_{E})$ that is not $\alpha$-stable, 
there is a $\sigma_{E}$-compatible reduction of structure $(E',\,\varphi')$ of $(E,\,\varphi)$ from
$(G,\,\rho)$ to $(G',\,\rho')$, such that $G'\, \subset\, G$ is a reductive complex algebraic
subgroup, and $(E',\,\varphi',\,\sigma_{E}|_{E'})$ is $\alpha$-stable.
\end{The}

Theorem \ref{th:JHreduction} will be proved after proving Proposition \ref{prop:JHpoly}.

To prove Theorem \ref{th:JHreduction} we shall follow the approach in \cite[Sections 2.9, 2.10, 2.11]{GGM:21}, checking 
that every construction done there is compatible with real structures. First, we establish a relation between 
polystable $\ssigma$-real $L$-twisted Higgs pairs, that are not stable, and a certain space of automorphisms. This
will be elaborated below.

Let $$\ad(E)\,=\, E(\gG)\, =\, E\times_G \gG\, \longrightarrow\, X$$ be the holomorphic vector bundle associated to $E$ 
for the adjoint action of $G$ on its Lie algebra $\gG$; so $\ad(E)$ is the adjoint vector bundle for $E$. Note that the 
fibers of $\ad(E)$ are Lie algebras isomorphic to $\gG$. More precisely, any fiber $E(\gG)_x$ is identified with $\gG$ 
uniquely up to an inner automorphism of $\gG$. Since the fibers of $\ad(E)$ are Lie algebras, the vector space 
$H^{0}(X,\, \ad(E))$ has the structure of a complex Lie algebra.

The involution $\sigma_{E}$ of $E$ and the involution $d\sigma_{G}$ of $\gG$ together define an
anti-holomorphic involution 
\begin{equation}\label{xe1}
\sigma_{\ad(E)}\, :\, \ad(E)\, \longrightarrow\, \ad(E)\, .
\end{equation}

The space of infinitesimal automorphisms of a principal $G$-bundle $E_G$ is given by
$H^{0}(X,\, \ad(E))$. In other words, $H^{0}(X,\, \ad(E))$ is the Lie algebra of the group
of all holomorphic automorphisms of $E$.
Hence the set of infinitesimal automorphism of a $L$-twisted Higgs pair 
$(E,\,\varphi)$ is given by
$$\aut(E,\,\varphi)\,:=\,\{s\,\in\, H^{0}(X,\, \ad(E))\,\mid\,
((d\rho)\otimes \text{Id}_L)(s)(\varphi)\,=\,0\}\, \subset\,
H^{0}(X,\, \ad(E))\,,
$$
where $d\rho$ is the homomorphism of Lie algebras associated to $\rho$ in \eqref{hrho}; note that
$d\rho$ induces a homomorphism of associated bundles. 
It is straight-forward to check that $\aut(E,\,\varphi)$ is a complex Lie subalgebra
of $H^{0}(X,\, \ad(E))$.

Let
\begin{equation}\label{sa}
\sigma_{\aut}\, :\, \aut(E,\,\varphi)\, \longrightarrow\, \aut(E,\,\varphi)
\end{equation}
be the conjugate linear involution defined by
$$\sigma_{\aut}(s)(\sigma_{X}(x))\,:=\, \sigma_{\ad(E)}(s(x))$$ for every $s\,\in\,
\aut(E,\,\varphi)$ and $x\,\in\, X$,
where $\sigma_{\ad(E)}$ is the anti-holomorphic involution of the adjoint bundle $\ad(E)$.

Note that given a section $s\,\in\, H^0(X,\, \ad(E))$, there is a unique Jordan decomposition
\begin{equation}\label{sd}
s\, =\, s^{ss} + s^{n}\, ,
\end{equation}
where
\begin{itemize}
\item $s^{ss}(x)\, \in\, \ad(E)_x$ is semisimple for all $x\, \in\, X$ while
$s^{n}(x)$ is nilpotent, and

\item $s^{ss}(x)$ and $s^{n}(x)$ commute for all $x\, \in\, X$.
\end{itemize}
(See \cite{BBN1}, \cite{BBN2}).
Let
$$
\aut^{ss}(E,\,\varphi)\,\subset\, \aut(E,\,\varphi)
$$
be the subspace consisting of all $s_1\,\in\, 
\aut(E,\,\varphi)$ such that $s_1(x)\, \in\, \ad(E)_x$ is semisimple for all
$x\, \in\, X$. Therefore, we have
$$s^{ss}\, \in\, \aut^{ss}(E,\,\varphi)\, ,$$
where $s^{ss}$ is the section in \eqref{sd}. In fact, if
$s\, \in\, \aut(E,\,\varphi)$ with $\sigma_{\aut}(s)\,=\, s$, then we have
$s^{ss}\, \in\, \aut^{ss}(E,\,\varphi)$ with $\sigma_{\aut}(s^{ss})\,=\, s^{ss}$.

Consider the Lie algebra $\aut(E,\,\varphi)$. Let
\begin{equation}\label{1c}
{\mathcal C}\, \subset\, \aut(E,\,\varphi)
\end{equation}
be a Cartan subalgebra such that
$$
\sigma_{\aut}({\mathcal C})\, =\, {\mathcal C}\, ,
$$
where $\sigma_{\aut}$ is constructed in \eqref{sa}. Note that
we have ${\mathcal C}\, \subset\, \aut^{ss}(E,\,\varphi)$.

This complex subalgebra $\mathcal C$ in \eqref{1c} produces a Levi subgroup $L(P)\, 
\subset\, P$ of a parabolic subgroup $P\, \subset\, G$, together with a holomorphic reduction 
of structure group of $E$ $$ E_{L(P)}\, \subset\, E $$ to the Levi subgroup $L(P)$ \cite[p.~55, 
Proposition 1.2]{BP} (see also \cite{BBN1}, \cite{BBN2}). We briefly recall below the 
construction of the above pair $(L(P),\, E_{L(P)})$.

Take any element $s\, \in\, \mathcal C$ such that the complex
subgroup of $\text{Aut}(E)$ generated by $\{\exp(ts)\}_{t\in \mathbb C}$
coincides with the subgroup generated by $\mathcal C$; elements of $\mathcal C$ satisfying this condition form a
nonempty Zariski open subset of $\mathcal C$.
Since the conjugacy classes of semisimple elements in $\mathfrak g$ is an affine variety, and $X$ is
compact, the condition that $s(x)\, \in\, \ad(E)_x$ is
semisimple for all $x\, \in\, X$ implies that the conjugacy class in $\mathfrak g$ determined by the element
$s(x)\, \in\, \text{ad}(E)_x$ is independent of $x$. Fix an element $s_0\, \in\, \mathfrak g$ in this conjugacy
class. Then $L(P)\, \subset\, $ is the centralizer, in $G$, of $s_0$ for the adjoint action of $G$.
Let
\begin{equation}\label{cs}
C(s)\, \subset\, \text{Ad}(E)
\end{equation}
be the sub-group scheme whose fiber over any $x\, \in\, X$ is the
centralizer, in $\text{Ad}(E)_x$, of $s(x)$ for the adjoint action on the Lie algebra; since the conjugacy
class of $s(x)$ is independent of $x\, \in\, X$, it follows that $C(s)$ is indeed a sub-group scheme of
$\text{Ad}(E)$. Also, recall that $\text{Ad}(E)$ is a quotient of $E\times G$, where two points
$(z_1,\, g_1), \, (z_2,\, g_2)\, \in\, E\times G$ are identified if there is $g\, \in\, G$ such that
$z_2\,=\, z_1g$ and $g_2\,=\, g^{-1}g_1g$. The complex submanifold
\begin{equation}\label{cs2}
E_{L(P)}\, \subset\, E
\end{equation}
is the locus of all point $z\, \in\, E$ such that the image of $(z,\, g)\, \in\, E\times L(P)$ in the quotient
space $\text{Ad}(E)$ lies in $C(s)$.

Now choose $s\, \in\, \mathcal C$ as above satisfying the extra condition that 
$\sigma_{\aut}(s)\,=\, s$ (as before, impose the condition that the subgroup of $\text{Aut}(E)$ generated by 
$\{\exp(ts)\}_{t\in \mathbb C}$ coincides with the subgroup generated by $\mathcal C$); since 
the subspace ${\mathcal C}^{\sigma_{\aut}}\, \subset\, \mathcal C$ fixed by $\sigma_{\aut}$ is 
Zariski dense in $\mathcal C$, such an element $s$ exists. Choose the above element $s_0\, \in\, 
\mathfrak g$ such that $d\sigma_G(s_0)\,=\, s_0$. This condition implies that 
$\sigma_G(L(P))\,=\, L(P)$ (the parabolic subgroup $P$ associated to $L(P)$ is not unique, and 
$P$ need not be preserved by $\sigma_G$). The antiholomorphic involution $\sigma_E$ of $E$ and 
the antiholomorphic involution $\sigma_G$ of $G$ together produce an antiholomorphic involution 
of $E\times G$. This involution descends to an antiholomorphic involution of $\text{Ad}(E)$. 
Since $\sigma_{\aut}(s)\,=\, s$, the sub-group scheme $C(s)$ in \eqref{cs} is preserved by this 
antiholomorphic involution of $\text{Ad}(E)$. Consequently, $E_{L(P)}$ in \eqref{cs2} is 
preserved by the antiholomorphic involution $\sigma_E$ of $E$.

We note that the condition that $(d\rho\otimes \text{Id}_L)(s)(\varphi)\,=\,0,$ for all $s\, \in\, \mathcal C,$
implies that $(E_{L(P)},\, \varphi)$ is, in fact, a reduction of structure group of $(E,\, \varphi)$.

Let
$$
\zZ\, \subset\, {\mathfrak g}
$$
be the center of $\mathfrak g$, so $\zZ$ is the Lie algebra of $Z$. Let
$$
E(\zZ)\ :=\, E\times^G \zZ
$$
be the vector bundle associated to $E$ for the adjoint action of $G$ on $\zZ$. Since the
adjoint action of $G$ on $\zZ$ is trivial, we have $E(\zZ)\ :=\, X\times\zZ$. Also, note that
$$
E(\zZ)\ \subset\,\text{ad}(E)
$$
is the fiberwise center of the Lie algebra bundle $\text{ad}(E)$.

While the infinitesimal automorphisms of $(E,\,\varphi)$ are parameterized by
$\aut(E,\,\varphi)$, the infinitesimal automorphisms of the triple $(E,\, \varphi,\, \sigma_E)$
constitute the subspace
$$
\aut(E,\, \varphi, \,\sigma_E)\, =\, \aut(E,\,\varphi)^{\sigma_{\aut}}\,:=
\{s \, \in\, \aut(E,\,\varphi)\, \mid\, \sigma_{\aut}(s)\,=\, s\}\, ,
$$
where $\sigma_{\aut}$ is the real involution of $\aut(E,\,\varphi)$ in \eqref{sa}. Therefore, we have
$$H^{0}(X,\,E(\zZ))^{\sigma_{\ad(E)}}\,=\, H^{0}(X,\,E(\zZ))\cap \aut(E,\, \varphi,\, \sigma_E)\, ;$$
note that the involution $\sigma_{\ad(E)}$ in \eqref{xe1} induces an involution of 
$H^0(X,\, \ad(E))$ (also denoted by $\sigma_{\ad(E)}$), and the subspace
$H^{0}(X,\,E(\zZ))\, \subset\, H^0(X,\, \ad(E))$ is preserved by $\sigma_{\ad(E)}$.

\begin{Pro}\label{pro:aut}
Let $(E,\,\varphi,\, \sigma_E)$ be a $\ssigma$-real $L$-twisted Higgs
pair. If $(E,\, \varphi,\, \sigma_E)$ is $\alpha$-stable then 
\begin{equation}\label{hypo2}
\aut(E,\,\varphi,\,\sigma_{E})\,=\,H^{0}(X,\, E(\mathfrak{z}))^{\sigma_{\ad(E)}} \,.
\end{equation}

Assume that $(E,\, \varphi,\, \sigma_E)$ is $\alpha$-polystable. Then
$(E,\varphi, \sigma_E)$ is $\alpha$-stable if and only if 
\begin{equation}\label{hypo1}
\aut^{ss}(E,\,\varphi,\,\sigma_{E})\,=\,H^{0}(X,\, E(\mathfrak{z}))^{\sigma_{\ad(E)}}\, .
\end{equation}
\end{Pro}

\begin{proof}
We will start proving the second part, assuming the first part. The first part will be proved
subsequently.

Assume that $(E,\, \varphi,\, \sigma_E)$ is $\alpha$-polystable. If
$(E,\, \varphi,\, \sigma_E)$ is $\alpha$-stable, then the first part of the proposition implies
that $\aut^{ss}(E,\,\varphi,\,\sigma_{E})\,=\,H^{0}(X,\, E(\mathfrak{z}))^{\sigma_{\ad(E)}}$, since
any section $s\, \in\, H^{0}(X,\, E(\mathfrak{z}))$ is pointwise semisimple. Note that
$H^{0}(X,\, E(\mathfrak{z}))\,=\, \mathfrak{z}$, because $H^{0}(X,\, E(\mathfrak{z}))$,
as noted before, is the trivial
holomorphic vector bundle on $X$ with fiber $\mathfrak{z}$. The ${\sigma_{\ad(E)}}$-invariant part
$H^{0}(X,\, E(\mathfrak{z}))^{\sigma_{\ad(E)}}\, \subset\, H^{0}(X,\, E(\mathfrak{z}))$ coincides
with $\mathfrak{z}^{d\sigma_G}$, where $d\sigma_G$ is, as above, the involution of $\mathfrak g$ induced
by $\sigma_G$.

On the other hand, if $(E,\, \varphi,\, \sigma_E)$ is not $\alpha$-stable, then for any reduction
$$
E_{L_s}\, \subset\, E_{P_s}\, \subset\, E
$$
as in \eqref{els}, consider the subbundle
$$
{\mathcal Z}(E_{L_s})\, \subset\, \text{ad}(E_{L_s})
$$
defined by the centers of the fibers of the adjoint bundle $\text{ad}(E_{L_s})$ of
$E_{L_s}$; so for any $x\, \in\, X$, the fiber ${\mathcal Z}(E_{L_s})_x$ is the center
of the Lie algebra $\text{ad}(E_{L_s})_x$. Note that ${\mathcal Z}(E_{L_s})$ is a trivial
vector bundle on $X$ whose fibers are identified with the center $\zZ_{L_{s}}$ of the Lie algebra of
the Levi subgroup $L_s$. Moreover,
$$
\zZ_{L_{s}}\,=\, H^0(X,\, {\mathcal Z}(E_{L_s}))\, \subset\, H^0(X,\, \text{ad}(E_{L_s}))
\, \subset\, \aut(E,\,\varphi)\, ,
$$
and the automorphism $\sigma_{\aut}$ of $\aut(E,\,\varphi)$ preserves the
above subspaces $H^0(X,\, \text{ad}(E_{L_s}))$ and $H^0(X,\, {\mathcal Z}(E_{L_s}))$; in fact, the
action of $\sigma_{\aut}$ on $H^0(X,\, {\mathcal Z}(E_{L_s}))$ coincides
with the action of $d\sigma_G$ on $\zZ_{L_{s}}$. The vector space $H^0(X,\, {\mathcal Z}(E_{L_s}))$
is strictly larger than $H^{0}(X,\, E(\mathfrak{z}))$, because $P_s$ is a proper parabolic subgroup
of $G$ (recall that $(E,\, \varphi,\, \sigma_E)$ is $\alpha$-polystable but not $\alpha$-stable).
Also, every element $s\, \in\, {\mathcal Z}(E_{L_s})$ is pointwise semisimple.
These together imply that $H^0(X,\, {\mathcal Z}(E_{L_s}))^{\sigma_{\ad(E)}}$ is strictly larger than
$H^{0}(X,\, E(\mathfrak{z}))^{\sigma_{\ad(E)}}$. This proves the second part of the proposition
assuming the first part.

To prove the first part of the proposition, we first note that
$$
H^{0}(X,\, E(\mathfrak{z}))^{\sigma_{\ad(E)}}\, \subset\, \aut(E,\,\varphi,\,\sigma_{E})\, ,
$$
because $H^{0}(X,\, E(\mathfrak{z}))\, \subset\, \aut(E,\,\varphi)$.

Let $(E,\,\varphi,\, \sigma_E)$ be a $\ssigma$-real $L$-twisted Higgs pair
which is $\alpha$-stable.

First assume that 
\begin{equation}\label{a1}
\aut^{ss}(E,\,\varphi,\,\sigma_{E}) \setminus H^{0}(X,\, E(\mathfrak{z}))^{\sigma_{\ad(E)}}\, \not=\,
\emptyset\, .
\end{equation}
Take an element $s\, \in\, \aut^{ss}(E,\,\varphi,\,\sigma_{E})$ that generates a Cartan subalgebra
of $\aut^{ss}(E,\,\varphi)$. As we saw in \eqref{cs2}, this section $s$
produces a reduction of structure group
$$
E({\mathbb L}^s) \, \subset\, E
$$
of $E$ to a Levi subgroup ${\mathbb L}^s\, \subset\, G$ which satisfies the following two conditions:
\begin{enumerate}
\item $E({\mathbb L}^s)$ is a reduction of structure group of $(E,\, \varphi)$, and

\item $\sigma_E(E({\mathbb L}^s))\,=\, E({\mathbb L}^s)$.
\end{enumerate}
{}From the assumption in \eqref{a1} it follows that the conjugacy class in $\mathfrak g$ determined
by $s$ does not lie in $\zZ$. Consequently, ${\mathbb L}^s\, \subsetneq\, G$ is a proper
parabolic subgroup.

The reduction $E({\mathbb L}^s) \, \subset\, E$ thus contradicts the given condition that
$(E,\,\varphi,\, \sigma_E)$ is $\alpha$-stable. In view of this contradiction, we conclude that
$$
\aut^{ss}(E,\,\varphi,\,\sigma_{E})\,=\, H^{0}(X\,, E(\mathfrak{z}))^{\sigma_{\ad(E)}}\, .
$$

To complete the proof we need to show that $\aut(E,\,\varphi,\,\sigma_{E})$ does not have any
nonzero nilpotent element.

Let $s\, \in\, \aut(E,\,\varphi,\,\sigma_{E})$ be a nonzero nilpotent element. 
This defines 
a parabolic subalgebra bundle
\begin{equation}\label{cp}
{\mathcal P}\, \subset\, \text{ad}(E)\, 
\end{equation}
constructed as follows. Since there are only finitely many
conjugacy classes of nilpotent elements of $\mathfrak g$, there is a open subset
$U\, \subset\, X$ such that
\begin{itemize}
\item $X\setminus U$ is a finite subset, and

\item the conjugacy class in $\mathfrak g$ determined by $s(x)\, \in\, \text{ad}(E)_x$
is independent of $x\, \in\, U$.
\end{itemize}

Take any $x\, \in\, U$. Let ${\mathfrak n}_1\, \subset\, \text{ad}(E)_x$ be the normalizer
of ${\mathbb C}\cdot s(x)$, and let ${\mathfrak r}_1$ be the nilpotent radical of ${\mathfrak n}_1$.
Now inductively define ${\mathfrak n}_{i+1}$ to be the normalizer of ${\mathfrak r}_i$ in
$\text{ad}(E)_x$ and define ${\mathfrak r}_{i+1}$ to be the nilpotent radical of
${\mathfrak n}_{i+1}$. Now $\{{\mathfrak n}_j\}_{j\geq 1}$ is an increasing sequence of
subspaces that converges to a parabolic subalgebra of $\text{ad}(E)_x$ (see
\cite[p.~340, Lemma 3.7]{AzB}). This way we obtain a parabolic subalgebra bundle
$${\mathcal P}'\, \subset\, \text{ad}(E)\vert_U$$
over $U$. Since the conjugates of a parabolic subalgebra in a complex reductive algebra
are parameterized by a complete variety (conjugates of a parabolic subalgebra
$\text{Lie}(P)\, \subset\, \mathfrak g$ are parameterized by $G/P$ which is a complete
variety), the above subalgebra bundle ${\mathcal P}'$ extends to a parabolic subalgebra bundle
$\mathcal P$ over $X$ (as in \eqref{cp}); see the proof of Proposition 3.9 in \cite{AzB}.
It is straight-forward to check that this subalgebra bundle $\mathcal P$ contradicts the
given condition that $(E,\,\varphi,\, \sigma_E)$ is $\alpha$-stable. This proves the proposition.
\end{proof}

The following two lemmas correspond to \cite[Lemma 2.16]{GGM:21} and
\cite[Lemma 2.17]{GGM:21} respectively. Note that here we only consider
reductions of structure group that are compatible with the real structure.

\begin{Lem}\label{lem:realreduction}
Let $(E,\,\sigma_{E})$ be a $(\sigma_{X},\sigma_{G},c)$-real $G$-bundle
over $(X,\,\sigma_X).$ Let $G'$ be a Lie subgroup of $G$ such that the Lie algebra
$\gG'\,=\,\Lie(G')$ has the following property: The normalizer $N_{G}(\gG')\, \subset\,
G$ for the adjoint action of $G$ on $\gG$ is $G'$ itself. The reductions
$E_{G'}\, \subset\, E$ of structure group of $E$ from $G$ to $G'$ such that
\begin{equation}\label{condition11}
 \sigma_{\Ad(E)}(\Ad (E_{G'}))\,=\, \Ad (E_{G'})\,,
\end{equation}
are in one-to-one correspondence with subbundles $F\,\subset\, \ad(E)$ of Lie algebras satisfying
the following two conditions:
\begin{enumerate}
\item $\sigma_{\ad(E)}(F)\,=\,F$, and

\item for any $x\,\in\, X$ and any trivialization $E_{x}\,\simeq\, G$
of the fiber $E_x$ (obtaining by fixing an element of $E_{x}$),
the subalgebra $F_{x}\, \subset\, \ad(E)_{x}$ is conjugate to $\gG'$, via the
trivialization $\ad(E)_{x}\,\simeq \,\gG$ induced by the trivialization of $E_x$.
\end{enumerate}
\end{Lem}

\begin{proof}
It is a consequence of \cite[Lemma 2.16]{GGM:21} and the fact that
equation (\ref{condition11}) is equivalent to the above condition (1) on $F$.
\end{proof}

Let $P\,\subset\, G$ be a parabolic subgroup. Let $E_P\, \subset\, E$ be a holomorphic
reduction of structure group such that 
\begin{equation}\label{condition1}
 \sigma_{\Ad(E)}(\Ad (E_{P}))\,=\,\Ad(E_P)\,. 
\end{equation}
Let $L_P\,\subset \,P$ be a Levi subgroup of $P$, and let $U\, \subset\, P$ be the unipotent
radical. The Lie algebras of $L_P$, $P$ and $U$ will be denoted by
$\lL$, $\pP$ and $\uU$ respectively.
Let $$E_{L}\, \subset\, E_P$$ be a holomorphic reduction of the structure group of $E_P$ from $P$
to $L_P$, such that
\begin{equation}\label{condition2}
 \sigma_{\Ad(E)} (\Ad(E_{L}))\, =\, \Ad(E_{L})\,.
\end{equation}
 
Conditions (\ref{condition1}) and (\ref{condition2}) imply that
$\sigma_{\ad(E)}\,=\,d(\sigma_{\Ad(E)})$ preserves all three subbundles $E_{P}(\uU)$, $~E_{P}(\pP)$ and
$E_{P}(\lL)$ of $\ad(E)$. Denote by
$\sigma_{E_{P}(\mathfrak{u})}$, $~\sigma_{E_{P}(\mathfrak{p})}$ and
$\sigma_{E_{P}(\mathfrak{l})}$ the real structures on
$E_{P}(\mathfrak{u})$, $~E_{P}(\mathfrak{p})$ and
$E_{P}(\mathfrak{l})$ respectively induced by $\sigma_{\ad(E)}$. There is a short exact sequence
of holomorphic vector bundles with real structure
\begin{equation}\label{eq:exactsequence}
\xymatrix{
0\ar[r] &(E_{P}(\mathfrak{u}),\, \sigma_{E_{P}(\mathfrak{u})})\ar[r] &(E_{P}(\mathfrak{p}),
\, \sigma_{E_{P}(\mathfrak{p})})\ar[r] &(E_{P}(\mathfrak{l}),\,
\sigma_{E_{P}(\mathfrak{l})})\ar[r] &0.}
\end{equation}
Note that \eqref{eq:exactsequence} is a short exact sequence of Lie algebra bundles, and
all the homomorphisms in \eqref{eq:exactsequence} are Lie algebra structure preserving. A
Lie algebra bundle (right) splitting of \eqref{eq:exactsequence} is a homomorphism of Lie
algebra bundles $f\, :\, E_{P}(\mathfrak{l})\, \longrightarrow\, E_{P}(\mathfrak{p})$ such that the
composition of homomorphisms
$$
E_{P}(\mathfrak{l})\, \stackrel{f}{\longrightarrow}\, E_{P}(\mathfrak{p})\, \longrightarrow\,E_{P}(\mathfrak{l})
$$
coincides with the identity map of $E_{P}(\mathfrak{l})$, where $E_{P}(\mathfrak{p})\, \longrightarrow\,E_{P}(\mathfrak{l})$
is the projection in \eqref{eq:exactsequence}. If
$$
f\circ \sigma_{E_{P}(\mathfrak{l})}\,=\, \sigma_{E_{P}(\mathfrak{p})}\circ f\, ,
$$
then the splitting $f$ is called $(\sigma_{E_{P}(\mathfrak{l})},\, \sigma_{E_{P}(\mathfrak{p})})$-real.

\begin{Lem}\label{Pro:equivariant} Let $E_P\, \subset\, E$ be a holomorphic reduction
of structure group to $P$ satisfying (\ref{condition1}). Reductions $E_{L}$ of the
structure group of $E_{P}$ from from $P$ to $L_P$ satisfying
(\ref{condition2}) are in a natural bijective
correspondence with the holomorphic Lie algebra bundle (right) splittings of the exact sequence in
(\ref{eq:exactsequence}) that are $(\sigma_{E_{P}(\mathfrak{l})},\,
\sigma_{E_{P}(\mathfrak{p})})$-real. 
\end{Lem}

\begin{proof}
Setting $G\,=\, P$ and $G'\,=\,L_P$ in Lemma \ref{lem:realreduction},
we conclude that the space of reductions of structure group of
$E_P$ from $P$ to $L_P$ satisfying (\ref{condition2})
is in bijective correspondence with
subbundles $F\,\subset\, E_{P}(\pP)$ such that
$\sigma_{\ad(E)}(F)\,=\,F$ and $F_{x}$ is conjugate to $\lL\,,$ once we identify
$E_{P}(\pP)_x$ with $\pP$ for any $x\,\in\, X$, by fixing an element of $(E_P)_x$.
\end{proof}

Let $(E,\varphi,\, \sigma_E)$ be a $\alpha$-polystable $\ssigma$-real $L$-twisted Higgs pair
that is not $\alpha$-stable. From Proposition \ref{pro:aut} it follows that
there exists a non-central element $\eta\,\in\,\aut(E,\varphi,\,\sigma_{E})$, in other words there exists
$\eta\,\in\, H^{0}(X, \, E([\mathfrak{g},\mathfrak{g}]))$ such that $$\sigma_{\aut}(\eta)\,=\,\eta\,.$$
There is an element
$$
u \, :=\,  u_r +\sqrt{-1}u_i\, \in\, [{\mathfrak g},\, {\mathfrak g}]\, ,
$$
where $u_r,\, u_i\, \in\, \kK\,=\, \text{Lie}(K)$ and $[u_r,\, u_i]\,=\, 0$, such that
$\eta(x)\, \in\, E([\mathfrak{g},\mathfrak{g}])_x$ lies in the conjugacy class of $u$ for every $x\, \in\, X$;
as before, $K$ is a maximal compact subgroup of $G$ preserved by $\sigma_G$ (see Proposition \ref{max}).
This element $u \, :=\,  u_r +\sqrt{-1}u_i$ satisfies the equation $d\sigma_{G}(u)\,=\, u$.
Let $a\,\in \,[\kK,\,\kK]$ be an infinitesimal generator of the torus generated by $u_i$ and $u_r$.
Then, $a$ also satisfies the equation $d\sigma_{G}a\,=\,a$.
Furthermore, the subgroup $$K_1\,:=\,Z_{K}(a)\,=\, \{h\,\in\, K \, \mid\, \Ad(h)a\,=\,a\}$$ is preserved by 
$\sigma_{G}$, because $K$ is preserved by $\sigma_G$. We denote by $G_1$ the complexification $K_{1}^{\CC}$ and by
$\sigma_{G_{1}}$ the restriction $\sigma_{G}|_{K_{1}^{\CC}}\,.$
The section $\eta\,\in\, H^{0}(X,\,\ad(E))$ induces a $G$-equivariant map
$$\phi_{\eta}\,:\,E\,\longrightarrow\,\mathfrak{g}$$ which is also $(\sigma_{E},\,
d\sigma_{G})$-real, because $\sigma_\aut(\eta)\,=\,\eta$. Therefore, we conclude that
$$E_1\,:=\,\{\,e\,\in \,E\,\mid\,\, \phi_{\eta}(e)\,=\,u\}$$
is preserved by $\sigma_E$, and $E_{1}$ is a $(\sigma_X,\sigma_{G_{1}},c)$-real
principal $G_{1}$-bundle. Since $\rho$ is a
$(\sigma_{G},\sigma_{\VV})$-compatible representation, the complex vector subspace
$$\mathbb{V}_1\,:=\,\{v\,\in\,\mathbb{V}\,\mid \, \rho(a)v\,=\,0\}$$ is equipped with
a real structure $\sigma_{\mathbb{V}_1}\,=\,\sigma_{\mathbb{V}}|_{\mathbb{V}_1}$, and
$$\rho_1\,:=\,\rho|_{K_1}$$ 
is a $(\sigma_{G_1},\sigma_{\VV_{1}})$-compatible representation.

\begin{Pro}\label{prop:JHpoly} The above $(\sigma_{X}, \sigma_{G_{1}}, c,\sigma_L,\sigma_{\mathbb{V}_1},\pm)$-real
$L$-twisted Higgs pair 
$$(E_1,\,\varphi_1\,:=\,\varphi|_{E_1},\, \sigma_{E_1}\,:=\,\sigma_E|_{E_1})$$ is $\alpha$-polystable.
\end{Pro}

\begin{proof}
First, we shall prove that $(E_1,\,\varphi_1,\,\sigma_{E_1})$
is $\alpha$-semistable. We denote by $\kK_{1}$ the Lie algebra of $K_{1}$. For $s\,\in \,\sqrt{-1}\, \kK_{1},$ consider
the parabolic subgroup 
$$P_{1,s}\,:=\,\{g\,\in\, G_{1}\,\mid\,\,
e^{ts}ge^{-ts}~\text{ is bounded when }~t\rightarrow\infty\}\,,$$
and the Levi subgroup
$$L_{1,s}\,:=\,\{g\,\in\, P_{1,s}\,\mid\,\,
\Ad(g)s\,=\,s\}\,.$$ 

Any holomorphic reduction of structure group $\sigma_1$ of $E_1$ to $P_{1,s}$ such that $$\sigma_{\Ad(E)}(\Ad E_{P_{1,s}})
\,=\,\Ad E_{P_{1,s}}$$ can be extended to a real reduction of structure group $\sigma$ to $P_s$ such that
$$\sigma_{\Ad(E)}(\Ad E_{P_{s}})\,=\,\Ad E_{P_{s}}\, .$$ Moreover, we have $$\deg E(\sigma,s)
\,=\,\deg E_{1}(\sigma_{1},s)\,.$$ 
Therefore, using the given condition that $(E,\,\varphi,\,\sigma_{E})$ is
$\alpha$-semistable it is deduced that $(E_1,\,\varphi_1,\,\sigma_{E_1})$ is also $\alpha$-semistable. 

From Lemma \ref{Pro:equivariant}, it follows that $(E_1,\,\varphi_1,\,\sigma_{E_1})$ is $\alpha$-polystable
if for every $s\,\in\, \sqrt{-1} \kK$, and for every holomorphic reduction
of structure group $\sigma$ of $E_1$ to $P_s$ that satisfies the two conditions
$\sigma_{\Ad(E_{1})}(\Ad E_{P_{1,s}})\, =\, \Ad E_{P_{1,s}}$ and 
$$\deg E_{1}(\sigma_1,s)-B(s,\alpha)\, =\, 0\,,$$ there is a real splitting $\omega_1$ of
the following exact sequence of real vector bundles
$$
\xymatrix{
0\ar[r] &(\,{E_{1}}_{\sigma_{1}}(\mathfrak{u_1}),\sigma_{{E_{1}}_{\sigma_{1}}(\mathfrak{u_1})})
\ar[r] &(\,{E_{1}}_{\sigma_{1}}(\mathfrak{p_{1}}),\sigma_{{E_{1}}_{\sigma_{1}}(\mathfrak{p_{1}})})
\ar[r] &(\,{E_{1}}_{\sigma_{1}}(\mathfrak{l_{1}}), \sigma_{{E_{1}}_{\sigma_{1}}(\mathfrak{l_{1}})})\ar[r]&0\,.}
$$
The argument for the existence of the splitting is identical to the argument for the non-real case.
(See \cite[Proposition 2.18]{GGM:21} for details).

On the other hand, we have $\varphi\,\in\, H^{0}(X,\, E_{L_s}(\mathbb{V}^{0}_{s})\otimes L)\,$
if and only if $$\rho(\omega(\psi_{\sigma,s}(\varphi)))\,=\,0\, ,$$
where $\omega$ is a splitting of (\ref{eq:exactsequence}) and $\psi_{\sigma_1,s}\,\in\, H^{0}(X,\, E_{\sigma}(\pP))$ is 
the section which is equal to $s$ on the fibers. By the same reasoning,
$ \varphi_1 \,\in\, H^{0}(X,\,E_{1\,L_{1,s}}(\VV_1)\otimes L) $ if and only if
\begin{equation}\label{eq:conditionvarphi}
 \rho(\omega_{1} (\psi_{\sigma_1,s}(\varphi)))\,=\,0\,,
\end{equation}
where $\psi_{\sigma_1,s}\,\in\, H^{0}(X,\, E_{1\,\sigma_{1}}(\pP_1))$ is the section
which is equal to $s$ on the fibers.
If we decompose $\pP$ and $\lL$, using characters $\eta \,\in\, \Hom(T,\,S^{1})$, where
$T\,\subset\, H$ is a maximal torus,
then $\omega_1$ is the restriction of $\omega$, which corresponds to setting $\eta\,=\,1$. 

Since $(E,\,\varphi)$ is $\alpha$-polystable, we have
$\varphi\,\in\, H^{0}(X,\, E_{L_s}(\mathbb{V}^{0}_{s})\otimes L)$. As noted before, this implies
that $\rho(\omega (\psi_{\sigma,s}(\varphi)))\,=\,0\,$
and hence \eqref{eq:conditionvarphi} holds. It was observed earlier that
$$\varphi_1 \,\in\, H^{0}(X,\, E_{1\,L_{1,s}}(\VV_1)\otimes L)$$
if \eqref{eq:conditionvarphi} holds.
\end{proof}

\begin{proof}[{Proof of Theorem \ref{th:JHreduction}}]
From Proposition \ref{prop:JHpoly}, we obtain a $\sigma_{E}$-compatible reduction of structure group
of a polystable $\ssigma$-real $L$-twisted Higgs pair $(E,\,\varphi,\,\sigma_{E})$ from $(G,\,\rho)$
to a subgroup $(K_{1}^{\CC},\, \rho_{1})$ with real structure, where $K_{1}^{\CC}\,\subsetneq\, G.$
We can iterate this process and finally, within a finite number $n$ of steps, we obtain a
$\sigma_{E}$-compatible holomorphic reduction of structure group
$(E_{n},\,\varphi_{n})$ of $(E,\,\varphi)$ from
$(G,\,\rho)$ to $(K_{n}^{\CC},\,\rho_{n})$, where $K_{n}^{\CC}$ is a Levi subgroup
of $G$ and $(E_{n},\,\varphi_{n},\,\sigma_{E}|_{E_{n}})$ is $\alpha$-stable.
This proves Theorem \ref{th:JHreduction}.
\end{proof}

\begin{Rem}The uniqueness of the Jordan--H\"older reduction (up to conjugation)
follows immediately from the uniqueness in the usual (with no real structures) 
case (see \cite[Proposition 2.20]{GGM:21}).
\end{Rem}

\subsection{Stability implies the existence of solution}\label{Subsection3.6}

In this section, we prove one implication of Theorem \ref{th:Hitchin--Kobayashi} for $\ssigma$-real $L$-twisted 
$G$-Higgs pairs $(E,\,\varphi,\, \sigma_E)$ that are $\alpha$-stable.

Let $\EE$ be a $C^\infty$ principal
$G$-bundle over $X$ equipped with a $(\sigma_X,\sigma_G,c)$-real structure $\sigma_\EE$ (see Section \ref{real-riemann}).
Let $h$ be a $\sigma_\EE$-compatible $C^\infty$
reduction of structure group of $\EE$ from $G$ to the maximal compact subgroup
$K$. By Proposition \ref{pr:dolbeautyconexiones}, there is a bijective correspondence between $\sigma_{\EE}$-compatible 
connections on $\EE_\K$ and $\sigma_{\EE}$-real complex structures 
on $\EE$. As a consequence, after fixing a $\sigma_L$-compatible $\U(1)$-connection on $L$, we
conclude that the space of all $\ssigma$-real $L$-twisted 
Higgs pairs $(E,\,\varphi, \,\sigma_E)$ such that the underlying $C^\infty$-bundle to $E$ is $\EE$ is
in bijective correspondence with 
the space of all triples $(\sigma_\EE,\,A,\,\varphi)$ such that
\begin{itemize}
\item $\sigma_\EE$ is a $(\sigma_X,\,\sigma_{G},\,c)$-real structure on $\EE$,

\item $A$ is a $\sigma_{\EE}$-compatible connection on $\EE_{K}$, and

\item $\varphi$ is a holomorphic section of $\EE(\VV)\otimes L$ satisfying 
$$\sigma_{V}\otimes\sigma_L(\varphi)\,=\,\pm\sigma_{X}^{*}(\varphi)\, ,$$
where $\sigma_V$ is the involution of $V\,=\,\EE(\VV)$ induced by $\sigma_{\EE}$ and $\sigma_{\VV}$, and $\sigma_L$ is the fixed 
involution on $L$. Here, abusing notation, we are denoting $L$ and its underlying smooth line bundle in the same way.
 \end{itemize}
We will denote by $\TTT$ the set of triples $(\sigma_\EE,A,\varphi)$ satisfying the above conditions. 

Let $\AAA$ be the space of connections on $\EE_{K}$. Let
$\SSS$ be the space of all $C^\infty$ sections of $V\otimes L$.
The product space
$$
\XXX\,:=\,\AAA\times\SSS
$$
is an
infinite-dimensional K\"ahler manifold equipped with a
Hamiltonian action of the
gauge group
\begin{equation}\label{ex}
\KKK\,=\,\Omega^0(E_{K}(K))\, .
\end{equation}
Take any $$\alpha\,\in\, \zZ(\kK)\subset \Lie(\KKK)\, .$$ In this case, the
moment map is given by
\begin{equation}\label{eq:momentmap}
\mu^{\alpha}(A,\varphi)\,:=\,\Lambda F_{A}+\mu(\varphi)-\sqrt{-1}\, \alpha
\end{equation}
for every $(A,\,\varphi)\,\in\, \XXX$, where $\mu(\varphi)$ is defined as in Equation (\ref{def:muh}).
As before,
$$
B\, \in\, \text{Sym}^2({\mathfrak g}^*)^G
$$
is a $K$-invariant non-degenerate bilinear form on the Lie algebra ${\mathfrak g}\,=\,\Lie(G)$ which
is positive on $\mathfrak k$.
Let $\left\langle\,,\,\right\rangle$ be the bilinear form on $\Lie(\KKK)$ induced by $B$, where $\KKK$
is constructed in \eqref{ex}.
The maximal weight of the action of 
$s\,\in\,\Lie(\KKK)$ on $(A,\,\varphi)\,\in\, \XXX$ is defined as follows:
\begin{equation}\label{eq:limite1}
\lambda^{\alpha}((A,\,\varphi),\,s)\,:=\,\lim_{t\rightarrow \infty} \lambda^{\alpha}_{t}((A,\varphi),s)\,,
\end{equation}
where 
\begin{equation}\label{eq:limite}
\lambda^{\alpha}_{t}((A,\varphi),s)\, :=\,\left\langle \mu^{\alpha}(e^{\sqrt{-1}t\,s}A \,,\,e^{\sqrt{-1}t\,s}\varphi)\,,\,
s\right\rangle\, .
\end{equation}

The limit in (\ref{eq:limite}) exists by Lemma 2.1.2 of \cite{MIR:981}. The integral of the moment map
in \eqref{eq:momentmap} is defined by
\begin{equation}\label{eq:integralmomentmap}
 \Psi^{\alpha}((A,\,\varphi),\,e^{\sqrt{-1} s})\,:=\,
\int_{0}^{1}\lambda^{\alpha}_{t}((A,\varphi),s)\,dt\,.
\end{equation}

\begin{Rem}\label{rem:propfundintegralmoment}
If we fix the pair $(A,\, \varphi)$ in \eqref{eq:integralmomentmap}, then $e^{\sqrt{-1}s}
\,\in\,\Hh$ is a critical point of $\Psi^{\alpha}$ if and only if $e^{\sqrt{-1}s}(A,\,
\varphi)$ is a solution of $\mu^{\alpha}(A,\,\varphi)\,=\,0,$ where $\mu^{\alpha}$ is defined in \eqref{eq:momentmap}.
\end{Rem}

\begin{Def}\label{def:section} Let $E$ be a holomorphic principal $G$-bundle $X$, and let $h$ be a
$C^\infty$ reduction of
structure group of $E$ from $G$ to $K$. Let $A$ be the Chern connection on the principal $K$-bundle $E_h$,
given by $h$, corresponding to the holomorphic structure on $E$.
Take $s\,\in\, \sqrt{-1}\kK$, and let $\sigma$ be a holomorphic
reduction of structure group of $E$ from $G$ to
$P_s$, where $P_s$ is the parabolic subgroup of $G$ defined by $s$. The section $\psi_{h,\sigma,s}\,\in\,
\Omega^{0}(E_h(\sqrt{-1}\kK))$ is defined as follows: the reduction $\sigma$ defines a holomorphic map \begin{equation}\label{def:sectionxi}
 \xi: E_{h}\,\longrightarrow\, G/P_{s}\,.
\end{equation}
If $e\,\in\, E_h$, then $\xi(e)\,=\,P$ is a parabolic subgroup of $G$ conjugate to $P_s$. 
From \cite[Lemma 2.6]{GGM:21}, using the antidominant character $\chi_s$, it follows that there exists
$s_{\xi(e)}\,\in\,\sqrt{-1}\kK$, such that $P=P_{s_{\xi(e)}}.$ The map
\begin{equation}\label{def:sectionpsi}
 \mapnormal{E_{h}}{\sqrt{-1}\kK}{e}{s_{\xi(e)}}{\psi}
 \end{equation}
induces a section $\psi_{h,\sigma,s}\,\in\,\Omega^{0}(E_h(\sqrt{-1}\kK)).$ 
\end{Def}

\begin{Pro}\label{pro:pesosygrados1}
Let $(E,\,\varphi,\,\sigma_E)$ be a $\ssigma$-real Higgs pair. Let $h$ be a $\sigma_E$-compatible
$C^\infty$ reduction of structure group of $E$ from $G$ to $K$. 
Let $(\sigma_\EE,\,A,\,\varphi)\,\in\, \TTT\, $ be the triple corresponding to $(E,\,\varphi,\,\sigma_E)$. Let
$\sigma$ be a holomorphic reduction of structure group of $E$ from $G$ to $P_s$. Then 
$$\deg E(\sigma,\,s)-B(s,\,\alpha)\,=\,\lambda^{\alpha}((A,\,\varphi),\, \sqrt{-1}\psi_{h,\sigma,s})\,,$$
where $\psi_{h,\sigma,s}\,\in\,\Omega^{0}(E_h(\sqrt{-1}\kK))$ is constructed above.
\end{Pro}

\begin{proof}
First, recall that the degree of a $\ssigma$-real Higgs pair is equal to the degree of its underlying Higgs pair 
(see \eqref{eq:deg}); also recall that the maximal weight of $(\sigma_\EE,\,A,\,\varphi)$ is the maximal weight
of $(A,\, \varphi)$. Therefore, it suffices to prove the proposition by forgetting real structures.

In \cite[Section 2.1.6]{MIR:981} it is proved that there is a bijective correspondence between pairs $(\sigma,\,s)$ and
filtrations
\begin{equation}\label{eq:filtration1}
V(\psi_{h,\sigma,s})
\end{equation}
of $V\,=\,E(\VV)$ constructed as follows: let $\lambda_1,\,\ldots,\, \lambda_r$ be the eigenvalues of
$\rho(-\sqrt{-1}\psi_{h,\sigma,s})$; then $E(\VV)^{\lambda_k}\,=\,\bigoplus_{j\leq k}E(\VV)(\lambda_j) $, where $E(\VV)(\lambda_j)$ are the eigenbundles of eigenvectors $\lambda_j.$

From \cite[Lemma 4.2]{MIR:98} we know that for the Hamiltonian action of $\KKK$ on $\AAA$, the maximal weight $\lambda( 
A,\, -\sqrt{-1}\psi_{\sigma,s})$ is equal to
\begin{equation}\label{eq:deg112}
 \lambda_r \deg(E(\VV))+\sum_{k=1}^{r-1}(\lambda_{k}-\lambda_{k-1})\deg(E(\VV)^{\lambda_k}).
\end{equation}
Equation (\ref{eq:deg112}) coincides with one of the different ways to calculate $\deg E(\sigma,s)$ defined in equation (\ref{eq:deg}) (see \cite[Lemma 2.12 (3)]{GGM:21}). 

Now, in a similar way, we can extend the Hamiltonian action of $\KKK$ on $\AAA\times\SSS$ to prove the proposition 
(see \cite[Lemma 4.3]{MIR:98}).
\end{proof}

Let $\GGG\,=\,\Omega^0(\EE(G))$ be the gauge group of $\EE$.
The real structures $\sigma_{X}$ and $\sigma_{G}$
together induce an involution $\sigma_{\GGG}$ on $\GGG$. Since $\sigma_G$ preserves $K$,
it follows that $\sigma_{\GGG}$ restricts to an involution $\sigma_{\KKK}$ on $\KKK$.
By abuse of notation, we denote also by $\sigma_{\GGG}$ the involution on $\GGG/\KKK$ and
the involution on $\Lie(\GGG).$ 

A triple $(\sigma_{\EE},\,A,\,\varphi)\,\in\, \TTT$ is called \textbf{simple} if there is 
no semisimple element $u\,\in\,\Lie(\GGG)$ such that
\begin{itemize}
\item $u$ is invariant under $\sigma_{\Lie(\GGG}$, and

\item $\left\langle \left\langle d\mu^{\alpha}(A,\varphi),\,\,u \right\rangle ,\,-\sqrt{-1} u\right\rangle\,=\,0$. 
\end{itemize}

\begin{The}\label{th:maintheo}
Take any $\alpha\,\in\, \zZ(\kK)$. Fix a simple $\ssigma$-real triple $(\sigma_{\EE},\,A,\,\varphi)$. 
There is an element $s\,\in\, \GGG/\KKK$ such that
\begin{enumerate}
\item $\sigma_\GGG (s)\,=\, s$, and

\item the function $\Psi^{\alpha}$ in \eqref{eq:integralmomentmap} attains its minimum at $s$.
\end{enumerate}
\end{The}

Theorem \ref{th:maintheo} will be proved after proving Proposition \ref{Pro:ctes}.

Fix from now on $p\,>\,2$. Let $(\sigma_{\EE},\,A,\,\varphi)$ be a simple triple whose 
corresponding $\ssigma$-real Higgs pair is $\alpha$-stable.
We define $$\Met_{2}^{p}\,:=\,L_{2}^{p}(E(\sqrt{-1}\, \kK)$$
to be the completion of 
$\Omega^{0}(E(\sqrt{-1}\, \kK))$ with respect to the Sobolev norm 
$$\norm{s}_{L_{2}^{p}}\,=\,\norm{s}_{L^{p}}+\norm{d_{A}s}_{L^{p}}+\norm{\nabla d_{A}s}_{L^{p}}\, ,$$
where $d_{A}s$ is the covariant derivative of $s$ with respect to $A$, and
$$\nabla\,:\,\Omega^0(T^*X\otimes \ad(\EE))\,\longrightarrow\, 
\Omega^1(T^*X\otimes \ad(\EE))$$
is the tensor product
of the Levi--Civita connection and $d_{A}$.
Note that $\Omega^{0}(E(\sqrt{-1}\, \kK))$ is actually isomorphic to $\GGG/\KKK$ using the
exponential map. Let $C$ be a positive real number. Consider the bounded metric space
$$\Met_{2\,C}^{p}\,:=\,\{s\,\in\,\Met_{2}^{p}\,\mid\, \norm{\mu^{\alpha} (e^{s}(A,\varphi))}_{L^{p}}^{p}\,\leq\, C\}\, .$$

\begin{Pro}
If a metric minimizes $\Psi^{\alpha}$ in $\Met_{2\,C}^{p}$, then it also minimizes $\Psi^{\alpha}$ 
in $\Met_{2}^{p}.$
\end{Pro}

\begin{proof} 
Suppose that $s\,\in\, \Met_{2\,C}^{p}$ is a minimum of $\Psi^{\alpha}$. Let $B\,=\,e^s(A)$ and $\Theta
\,=\,e^s(\varphi)$. We define the differential operator
$$L(u)\,:=\,\sqrt{-1}\,\frac{\partial}{\partial t}\mu(e^{tu}(B,\Theta))\big\vert_{t=0}\,\,,$$
for every $u\,\in\, L_{2}^{p}(E(\sqrt{-1}\kK))$.

By \cite[Lemma 3.4.2]{bradlow}, it suffices to prove that $\Ker(L)\,=\,0$, because if $s$ minimizes $\Psi^{\alpha}$ and 
$\Ker(L)\,=\,0$, then $\mu(B,\Theta)\,=\,0$, and by Remark \ref{rem:propfundintegralmoment}, one has that $s$ minimizes 
$\Psi^{\alpha}$ on $\Met_{2}^{p}.$
 
We prove that $\Ker (L)\,=\,0$, by contradiction. Let $0\,\neq\, u\,\in\, \Ker (L)$. Therefore,
the element $\frac{u+\sigma_{\Lie(\GGG)}(u)}{2}$ also lies in $\Ker(L)$. Then we have $$\left\langle 
-\sqrt{-1}\,L(\frac{u+\sigma_{\Lie(\GGG)}(u)}{2})\,, \,\frac{u+\sigma_{\Lie(\GGG)}(u)}{2} \right\rangle\,=\,0\, .$$ From 
\cite[Equation 2.13]{MIR:981} it follows that $\frac{u+\sigma_{\GGG}(u)}{2}$ leaves $(B,\,\Theta)$
invariant. Consequently, 
$(\sigma_{\EE},\,B,\,\Theta)$ is not simple and hence $(\sigma_\EE,\,A,\,\varphi)$ is not simple.
In view of this contradiction we conclude that $\Ker(L)\,=\,0$. 
\end{proof}

\begin{Pro}\label{Pro:ctes}
There exist positive constants $C_{1}$ and $C_{2}$ such that 
$$\sup|s|\,\leq\, C_{1}\Psi^{\alpha}((A,\varphi),e^{s})+C_2$$
for every $s\,\in\, \Met_{2\,C}^{p}$.
\end{Pro}

\begin{proof}
The given condition that $p\,>\,2$ implies that $L^p_2\,\hookrightarrow\, C^0$,
which follows from the Sobolev embedding theorem, and it makes sense to consider $\sup|s|$.
Repeating the same arguments as in \cite[Section 6]{MIR:98}, it suffices to prove that
\begin{equation}\label{eq:desigualdad}
\norm{s}_{L^{1}}\,\leq\, C_{1}\Psi^{\alpha}((A,\varphi),e^{s})+C_2\,.
\end{equation}

To prove \eqref{eq:desigualdad} by contradiction,
suppose that $C_1,\,C_2$ satisfying (\ref{eq:desigualdad}) do not exist. Then it can
be shown that there is a sequence $\{u_i\}\,\in\,
\Met_{2\,C}^{p}$ converging weakly to some 
$$u_{\infty}\,\in\, \Met_{2\,C}^{p}\, ,$$ such that 
\begin{equation}\label{eq:contradicts}
 \lambda^{\alpha}((A,\varphi),u_{\infty})\,\leq\, 0\,.
\end{equation}
Indeed, following \cite[Lemma 2.5.4]{MIR:981}, if $C_1,\,C_2$ do not exist, then there is a sequence of real
numbers $\{C_j\}$ and a sequence $\{s_{j}\}$ of metrics in $\Met_{2\,B}^{p}$ such that
$$\lim_{j\rightarrow\infty}C_{j}\,=\,\infty\ \ \text{ and }\ \ \norm{s_{j}}_{L^{1}}
\,\geq\, C_{j}\Psi^{\alpha}((A,\varphi),e^{s_{j}})\, .$$
Let $l_j\,:=\,\norm{s_{j}}_{L^{1}}$ and $u_j\,:=\,\frac{s_j}{l_j}$; then $\norm{u_{j}}_{L^{1}}
\,=\,1$ and $\sup |u_j|\,\leq\, C_{j}\,.$
 
Passing to a subsequence if necessary, we can suppose that $\{u_j\}$ converge to an $u_{\infty}\,,$
weakly on $L_{1}^{2}(E(\sqrt{-1}\kK))$, and $\lambda^{\alpha}((A,\varphi),u_{\infty})\,\leq\, 0.$ From
the definition of the integral of the moment map we have the following:
$$ 
\frac{1}{C_{j}}\,\geq\,\frac{\Psi^{\alpha}((A,\varphi),e^{s_{j}})}{\norm{s_j}}\,\geq\,
\frac{l_{j}-t}{l_j}\lambda^{\alpha}_{t}((A,\varphi),-\sqrt{-1} u_j)+\frac{1}{l_j}\int_{0}^{t}
\lambda^{\alpha}_{t}((A,\varphi),-\sqrt{-1} u_j)dt
$$
$$=\, \frac{l_{j}-t}{l_j}(\lambda^{\alpha}_{t}(A,-\sqrt{-1} u_j)+\lambda^{\alpha}_{t}(\varphi,-\sqrt{-1} u_j))
$$
$$
+\frac{1}{l_j}\int_{0}^{t}(\lambda^{\alpha}_{l}(A,-\sqrt{-1} u_j)+\lambda^{\alpha}_{l}(\varphi,-\sqrt{-1} u_j))dl\, .
$$
These inequalities together imply that $\norm{\overline{\partial}_{A}(u_j)}$ is bounded, so 
there is a subsequence $\{u_j\}$ converging to $u_{\infty}.$ In addition, we have
$$\lambda^{\alpha}((A,\varphi),u_{\infty})\,\,\leq\,\, \lim_{t\rightarrow 
\infty}\lambda^{\alpha}_{t}((A,\varphi),-\sqrt{-1} u_j) \,\leq\, 0\,.$$ The subsequence 
$\{\sigma_{\GGG}(u_i)\}\,\in\, \Met_{2\,C}^{p}$ converges weakly to 
$\sigma_{\GGG}u_{\infty}\,\in \,\Met_{2\,C}^{p},$ and
\begin{equation}
 \lambda^{\alpha}((A,\varphi),\,\sigma_{\GGG}(u_{\infty}))\,\leq\, 0\,.
\end{equation}
The partial sums of subsequences $\{\frac{u_i+\sigma_{\GGG}(u_i)}{2}\}$ converge to $$u'_{\infty}
\,=\,\frac{u'_{\infty}+\sigma_{\GGG}(u'_{\infty})}{2}\,\in\, \Met_{2\,C}^{p}$$
and this element is invariant under $\sigma_{\GGG}$.
Then $\rho(u'_{\infty})$ has real constant eigenvalues. The filtration $V(u'_{\infty})$ induces
a holomorphic reduction of structure group $\sigma$ from $G$ to $P$ 
such that $E_{P}\,=\,\sigma_{E}(E_{P}),$ and therefore
 $\sigma_{\Ad(E)}(\Ad E_{P_{s}})\,=\,\Ad E_{P_{s}}$. Since by hypothesis 
$(E,\varphi)$ is
$\alpha$-stable, it follows that $$\deg E(\sigma,s)+B(\alpha,s)\,>\,0\, .$$
In view of Proposition \ref{pro:pesosygrados1}, this inequality contradicts (\ref{eq:contradicts}), completing the proof.
\end{proof}

\begin{proof}[{Proof of Theorem \ref{th:maintheo}}]
To prove Theorem \ref{th:maintheo}, let $\{s_i\}$ be a minimizing sequence for $\Psi^{\alpha}$, then
$$\{s'_i\}\,=\,\{\frac{s_i+\sigma_{\GGG}(s_i)}{2}\}$$ is a $\sigma_{\GGG}$-invariant minimizing sequence for
$\Psi^{\alpha}$. From Proposition \ref{Pro:ctes}, after passing to a
suitable subsequence, we can suppose that $\{s'_i\}$ converges weakly to some
$s'$, invariant under $\sigma_{\GGG}$, that
minimize $\Psi^{\alpha}$.
\end{proof}

\begin{The}\label{th:stablecase}
Let $(E,\,\varphi,\, \sigma_E)$ be a $\alpha$-stable $\ssigma$-real $L$-twisted $G$-Higgs pair.
Then there is a $\sigma_E$-compatible $C^\infty$ reduction of structure group
$h$ from $G$ to $K$ satisfying \eqref{eq:Einstein}.
\end{The}

\begin{proof}
Fix a $\sigma_E$-compatible $C^\infty$ reduction of structure group of $E$ from $G$ to $K$. Let
$$(\sigma_{\EE},\, A,\,\varphi)\,\in\,\TTT$$ be the
triple associated to $(E,\,\varphi,\, \sigma_E)$. The existence of a $\sigma_E$-compatible $C^\infty$ reduction
of structure group $h$ of $E$ from $G$ to $K$ satisfying \eqref{eq:Einstein} is equivalent to the existence of
a $\sigma_{\EE}$-compatible pair $(A',\,\varphi')$ satisfying the equation
\begin{equation}\label{eq:condition}
\mu^{\alpha}(A',\varphi')\,=\,0\, ,
\end{equation}
where $\mu^{\alpha}$ is defined in (\ref{eq:momentmap}).

From Remark \ref{rem:propfundintegralmoment} we know that fixing a $\sigma_E$-compatible pair $(A,\,\varphi)$ in
equation \eqref{eq:integralmomentmap},
$$e^{\sqrt{-1} s}\,\in\, \Hh$$ is a critical point of $\Psi^{\alpha}$ if and only if $e^{\sqrt{-1} s}(A,\,\varphi)$ is a solution of (\ref{eq:condition}).

Finally, since $(E,\,\varphi,\, \sigma_E)$ is $\alpha$-stable, it follows that $(\sigma_{\EE},\, A,\,\varphi)$
is simple (see Proposition \ref{pro:aut} and \cite[Definition 3.8]{MIR:98}), and
hence we can apply Theorem \ref{th:maintheo} to complete the proof.
\end{proof}

\subsection{Polystability implies the existence of a solution}

In this section, and in the next one we prove Theorem \ref{th:Hitchin--Kobayashi}.

Let $(E,\,\varphi, \,\sigma_E)$ be a $\alpha$-polystable $\ssigma$-real $L$-twisted Higgs pair. If $(E,\,\varphi, \,\sigma_E)$ is $\alpha$-stable then it admits a $\sigma_{E}$-compatible Hermite-Einstein-Higgs reduction (see Theorem \ref{th:stablecase}). If $(E,\,\varphi, \,\sigma_E)$ is not $\alpha$-stable, by the Jordan-H\"older reduction (Theorem \ref{th:JHreduction}) we obtain a $(\sigma_{X},\sigma_{G}|_{G'},c,\sigma_{L}, \sigma_{\VV}|_{\VV'},\pm)$-real Higgs pair $(E',\varphi',\sigma_{E}|_{E'})$ that is $\alpha$-stable, then, by the previous case, it admits a $\sigma_{E'}$-compatible reduction of structure of $(E'_{K'},\varphi'_{K'})$ from $(G',\rho')$ to $(K',\rho|K')$ satisfying (\ref{eq:Einstein}).
This reduction defines another reduction $(E_K,\varphi_K)$ of $(E,\varphi)$ from $(G,\rho)$ to $(K,\rho|K)$, that, with respect to the structure group, is given by
$$E_K\,=\,E_{K'}\times_{K'}K\,=\,E'_{K'}\times_{K'}K$$
and with respect to the Higgs field, $\varphi_K$ is the image of $\varphi'_{K'}$ under the homomorphism 
$$H^0(X,E'_{K}(\VV'_{K})\otimes L)\,\longrightarrow\,
H^0(X,E_{K}(\VV_{K})\otimes L)$$
induced by the injection $\xymatrix{E'_{K'}(\VV_{K'})\ar@{^{(}->}[r] & E_{K}(\VV_{K})},$ where $\VV_K'$ and $\VV_K$ are complex vector spaces such that $\rho(K')=\GL(\VV_{K'})$ and $\rho(K)=\GL(\VV_{K})$, respectively.
Since the reduction $(E'_{K},\varphi'_{K'})$ is $\sigma_{E'}$-compatible and it is Hermite Einstein Higgs, then one can see that the reduction $(E_K,\varphi_K)$ is also $\sigma_E$-compatible and it is Hermite-Einstein-Higgs, in other
words, it satisfies Equation (\ref{eq:Einstein}). 

\subsection{Existence of a solution implies polystability}\label{Subsection3.7}

\begin{Pro}[{Existence of a solution implies semistability}]\label{prop.es}
Assume that there exist a $\sigma_{E}$-compatible Hermite--Einstein--Higgs reduction $h$ of a 
$\ssigma$-real $L$-twisted Higgs pair $(E,\,\varphi,\, \sigma_{E})$. Then
$(E,\,\varphi,\, \sigma_{E})$ is $\alpha$-semistable.
\end{Pro}

\begin{proof}
Suppose that $(E,\,\varphi,\, \sigma_{E})$ is not $\alpha$-semistable. Then, there is
a holomorphic reduction $\sigma$ of structure group of $E$ from $G$ to $P_s$
such that $\sigma_{\Ad(E)}(\Ad E_{P_{s}})\,=\,\Ad E_{P_{s}}$
and 
\begin{equation}\label{eq:nonestability}
 \deg E(s,\sigma)-B(s,\alpha)\,<\,0\,.
\end{equation}From Proposition \ref{pro:pesosygrados1}, one has that 
\begin{equation}\label{eq:contradict}
 \lambda^{\alpha}(\,(A,\varphi)\,,\sqrt{-1} \psi_{h,\sigma,s}\,)\,<\,0\,.
\end{equation}
By hypothesis $h$ is Hermite-Einstein-Higgs, then $\lambda_{0}^{\alpha}(\,(A,\,\varphi)\,,\sqrt{-1} \epsilon_{h,\sigma,s}\,)\,=\,0$. (See Equation \ref{eq:limite1}). Since $\lambda_{t}^{\alpha}(\,(A,\varphi)\,,\sqrt{-1} \psi_{h,\sigma,s}\,)$ is a non decreasing sequence, it follows that $$\lambda^{\alpha}(\,(A,\varphi)\,,\sqrt{-1} \psi_{h,\sigma,s}\,)
\,\geq\, 0$$ and this contradicts \eqref{eq:contradict}.
\end{proof}

\begin{Pro}[{Existence of solutions implies polystability}]
Assume that there exist a $\sigma_{E}$-compatible Hermite--Einstein--Higgs reduction $h$ of a
$\ssigma$-real $L$-twisted Higgs pair $(E,\,\varphi,\, \sigma_{E})$. Then
$(E,\,\varphi,\, \sigma_{E})$ is $\alpha$-polystable. 
\end{Pro}

\begin{proof}
Let $h$ be a $\sigma_{E}$-compatible Hermite--Einstein--Higgs reduction. Take
$s\,\in\, \sqrt{-1}\kK$, and let $\sigma$ be any reduction of structure group from $G$
to $P_{s}$, such that $\sigma_{\Ad(E)}(\Ad E_{P_{s}})\,=\,\Ad E_{P_{s}}$. Since $h$ is
$\sigma_{E}$-compatible, $$\sigma_{E}(E_{h})\,=\,E_{h}\, ,$$ where $E_{h}$ is the reduced 
$K$-bundle. We define $\xi'$ as the map such that the following diagram
\begin{equation}\label{con:diagram1}
 \xymatrix{
 E_{h}\ar[d]_{\sigma_{E}}\ar[rr]^{\xi} && G/P_{s}\ar[d]_{\sigma_{G}}\\
 E_{h}\ar[rr]^{\xi'} && G/P_{d\sigma_{G}(s)}\,,}
\end{equation} commutes, where $\xi$ is defined in (\ref{def:sectionxi}). Analogously, we define $\psi'$ 
\begin{equation}\label{con:diagram2}
 \xymatrix{
 E_{h}\ar[d]_{\sigma_{E}}\ar[rr]^{\psi} && \sqrt{-1}\kK \ar[d]_{d\sigma_{G}}\\
 E_{h}\ar[rr]^{\psi'} && \sqrt{-1}\kK\,,}
\end{equation} where $\psi$ is defined in (\ref{def:sectionpsi}).
The maps $\psi$ and $\psi'$ define sections $\psi_{h,\sigma,s}$ and $\psi'_{h,\sigma,\sigma_{G}s}$ in $E_{h}(\sqrt{-1}\kK)$ such that 
the associated reductive filtrations $V(\psi_{h,\sigma,s})$ and $V(\psi'_{h,\sigma,\sigma_{G}s})$ (see equation (\ref{eq:filtration1})) are related in the following way: they define 
reductions $E_{P}$ and $\sigma_{E}(E_{P})$
that satisfy $\sigma_{\Ad(E)}(\Ad E_{P_{s}})\,=\,\Ad E_{P_{s}}$.

The $\ssigma$-real Higgs pair $(E,\varphi, \sigma_{E})$ is $\alpha$-semistable by Proposition \ref{prop.es}. 
Let $\sigma$ be a reduction of
structure group of $E$ from $G$ to $P_s$ such that $\sigma_{\Ad(E)}(\Ad E_{P_{s}})\,=\,\Ad E_{P_{s}}$
as well as
\begin{equation}\label{eq:nonestability2}
 \deg E(s,\sigma)-B(s,\alpha)\,=\,0\,.
\end{equation} Let $A$ be the connection on $E_h$ that corresponds to the holomorphic structure on $E$.
From Proposition \ref{pro:pesosygrados1} and \eqref{eq:nonestability2}, it follows that
\begin{equation}
 \lambda^{\alpha}(\,(A,\varphi)\,,\sqrt{-1} \psi_{h,\sigma,s}\,)=0\,.
\end{equation}

Since $h$ is an Hermite-Einstein-Higgs reduction, then $\lambda_{0}^{\alpha}(\,(A,\varphi)\,,\sqrt{-1} 
\psi_{h,\sigma,s}\,)\,=\,0$ and since $\lambda_{t}^{\alpha}(\,(A,\varphi)\,,\sqrt{-1} \psi_{h,\sigma,s}\,)$ is a 
non-decreasing sequence, then $\lambda_{t}^{\alpha}(\,(A,\varphi)\,,\sqrt{-1} \psi_{h,\sigma,s}\,)\,=\,0,$ for all 
$t.$ This implies that $e^{t\psi_{h,\sigma,s}}$ fix $A$, for any $t$. Therefore, the filtration 
$V(\psi_{h,\sigma,s})$ induces a reduction of the structure group of $E_{P_s}$ from $P_s$ to $L_s$. Analogously, 
$e^{t\psi'_{h,\sigma,\sigma_{G}s}}$ fixes $\sigma_{\AAA}A$, for any $t$, where $\sigma_{\AAA}$ is the real structure 
induced in the space of connections by $\sigma_{E}$ and $\sigma_{X}.$ The filtration 
$V(\psi'_{h,\sigma,\sigma_{G}s})$ induce a reductions of the structure group of $\sigma_{E}(E_{P_s})$ from 
$\sigma_{G}P_s$ to $\sigma_{G}L_s.$ Both reductions are related by
$$\sigma_{\Ad(E)}(\Ad (E_{L_s}))\,=\,\Ad (E_{L_s})\,.$$
The element $e^{t\epsilon_{h,\sigma,s}}$ fixes also $\varphi$ for any $t$. This implies that 
$\varphi\,\in\, H^{0}(X,\,E_{L_s}(\mathbb{V}^{0}_{s})\otimes L).$ This completes the proof.
\end{proof}

\section{Real Higgs bundles and the non-abelian Hodge correspondence}\label{Section4}

Let $G^\R$ be a semisimple real form of a connected complex semisimple Lie group $G$ defined by a conjugation
$\mu\,\in\, \Conj(G)$. Let $\tau \,\in\, \Conj(G)$, commuting with $\mu$ and defining a compact real form $K\,
\subset \,G$. The subgroup $H^\R\,:=\,K\cap G^\R$ defines a maximal compact subgroup of $G^\R$. Let $\lieg^\R
\,=\,\lieh^\R\oplus \liem^\R$ be 
the Cartan decomposition of $\lieg^\R$, the Lie algebra of $G^\R$, where $\lieh^\R$ is the Lie algebra of 
$H^\R$ and $\liem^\R$ is its orthogonal complement with respect to the Killing form. Let $H$ and $\liem$ be the 
complexifications of $H^\R$ and $\liem^\R$, respectively.
Let $\sigma_{G}$ be a $(\mu,\tau)$-compatible conjugation of $G$
 (see Example \ref{isotropyrep}). By abuse of notation, we denote by $\sigma_{G}$ the restriction to $H$ and 
the restriction of $d\sigma_{G}$ to $\mM$. Then
the isotropy representation $\iota\,:\,H\,\longrightarrow\, \GL(\mM)\,$ is
$(\sigma_{G},\, \sigma_{G})$-compatible (see Example \ref{isotropyrep}).

Let $(X,\, \sigma_{X})$ be a compact Klein surface, and let
$K_{X}$ be the canonical bundle of $X$.
The anti-holomorphic involution $\sigma_X$ induces a real structure $\sigma_{K_{X}}$ on $K_{X}$. 

\subsection{Real $G^\R$-Higgs bundles and Hitchin equations}\label{Subsection4.1}

Let $ Z(H)$ be the centers $H$ and
$Z(H)_{2}^{\sigma_{G}}$ be the subgroup of elements of $Z(H)$ of order two 
that are invariant under $\sigma_{G}.$
Set 
\begin{equation}\label{eq:intersection}
 Z_\iota^{\sigma_{G}}\,:=\,Z(H)_{2}^{\sigma_{G}}\cap\ker(\iota)
\end{equation}
and let $c\,\in\, Z_\iota^{\sigma_{G}}$. 
A $(\sigma_{X},\sigma_{G},\, c,\,\pm)$-\textbf{real $G^\R$-Higgs bundle} is 
a $(\sigma_{X},\sigma_{G},c,d\sigma_{G},\sigma_{K_{X}},\pm)$-real $K_{X}$-twisted $H^{\CC}$-Higgs-pair of type $\iota$. Set $\ssigma=(\sigma_{X},\sigma_{G}, c,\pm)$. Sometimes we will refer to these objects as $\ssigma$-real 
$G$-Higgs bundles.

Let $\zZ(\hH)$ be the center of $\hH$ and $\alpha\,\in \,
\sqrt{-1}\hH\,\cap \,\zZ(\hH)$. The notions of $\alpha$-stability, $\alpha$-semistability 
and $\alpha$-polystability in Section \ref{Subsection3.3} apply to $\ssigma$-real $G^\R$-Higgs bundles. Since we are mostly interested in the relation of these objects to representations of the fundamental group of $X$ we will take 
$\alpha\,=\,0$ and will refer to $0$-stability simply as stability (the same for semistability and polystability).

Denote by $\Mm(G^\R,\sigma_{X},\sigma_{G}, c,\pm)$ the moduli space of isomorphism classes of polystable
$(\sigma_{X},\,\sigma_{G},\, c,\,\pm)$-real $G^\R$-Higgs bundles.
We may fix the topological type $d\in\,\pi_1(H)\cong \pi_1(H^\R)$, and 
consider the subvariety $\Mm_{d}(G^\R,\sigma_{X},\sigma_{G}, c,\pm)$. 

Let $(E,\,\varphi,\, \sigma_{E})$ be a $\ssigma$-real $G^\R$-Higgs bundle over 
$(X,\, \sigma_{X})$.
Let $h$ be a $\sigma_E$-compatible reduction of the structure group of $E$ from $H$ to $H^\R$,
and let 
\begin{equation}\label{def:definiciontau}
\tau_{h}\,:\,\Omega^{1,0}(E(\mM))\,\longrightarrow\,\Omega^{0,1}(E(\mM))
\end{equation}
be the map defined by
the compact conjugation $\tau$ in the fibers induced by $h$ combined with complex conjugation on complex $1$-forms. 

Choose a $\sigma_X$-compatible K\"ahler form on $X$ as defined in Section 
\ref{real-riemann}. This defines a $\sigma_{K_X}$-compatible metric on $K_X$. Taking $\alpha\,
=\,0$, one can show that the 
Hermite--Einstein--Higgs equation 
(\ref{eq:Einstein}) coincides in this particular case with the 
{\bf Hitchin equation}
\begin{equation}\label{eq:HEhiggs}
F_{h}-[\varphi,\, \tau_{h}(\varphi)]\,=\,0\,.
\end{equation}

As a particular case of Theorem \ref{th:Hitchin--Kobayashi} one thus has the following.

\begin{The}\label{th:HKreal}
A $\ssigma$-real $G^\R$-Higgs bundle $(E,\varphi, \sigma_E)$ is polystable if and only if
there exists a $\sigma_E$-compatible reduction $h$ of the structure group of $E$ from
$H$ to $H^\R$ satisfying (\ref{eq:HEhiggs}).
\end{The}

And Corollary \ref{Cor:realpairspairs} gives now the following.

\begin{Cor}\label{Cor:realpolypoly}
A $\ssigma$-real $G^\R$-Higgs bundle $(E,\,\varphi, \,\sigma_{E})$ is polystable if and only if the underlying
$G^\R$-Higgs bundle $(E,\,\varphi)$ is polystable.
\end{Cor}

From the point of view of moduli spaces it is convenient to look at the Hitchin equation from a different
but equivalent point of view (see \cite{GGM:21}). To explain this, fix a $C^\infty$ principal $H$-bundle 
$\EE_{H}$ and a reduction $h$ to a $H^\R$-bundle $\EE_{H^\R}\,\subset\, \EE_H$. 
We are then looking for a connection $A$ on $\EE_{H^\R}$ and a smooth section 
$\varphi\,\in\, \Omega^{1,0}(X,\,\EE_{H}(\liem))$. The Hitchin equations become then 

\begin{equation}\label{hitchin}
\begin{array}{l}
F_A -[\varphi,\tau_h(\varphi)]\,=\, 0\\
\dbar_A\varphi\,=\, 0\, .
\end{array}
\end{equation}
Here $d_A$ is the covariant derivative associated to $A$, and
$\dbar_A$ is the $(0,1)$ part of $d_A$. 

Let $\Mg(G^\R)$ be the space of solutions $(A,\varphi)$ to 
(\ref{hitchin}) modulo the gauge group $\HHH^\R$ of $\EE_{H^\R}$.
Recall (see \cite{GGM:21}) that, if $d\in\pi_1(H)$ is the topological class
of $\EE_H$, and $\cM_d(G^\R)$ is the moduli space of $G^\R$-Higgs bundles, 
then there is a homeomorphism
\begin{equation}\label{HK-Higgs-bundles}
\Mg(G^\R)\,\cong\, \cM_d(G^\R).
\end{equation} 

Now, if we are equipped with conjugations $\sigma_X$ and $\sigma_G$ 
and $c\,\in\, Z_\iota^{\sigma_G}$, 
as at the beginning of Section \ref{Section4}, we can consider a 
$(\sigma_X,\sigma_G,c)$-real structure $\sigma_{\EE_{H}}$ on $\EE_{H}$, which,
if the reduction $h$ is chosen to be $\sigma_{\EE_{H}}$-compatible, it will restrict to
$\sigma_{\EE_{H^\R}}$ on $\EE_{H^\R}$. 

Let $\TTT^\pm$ be the set triples $(\sigma_{\EE_{H}},\,A,\,\varphi)$, where 
$A$ is $\sigma_{\EE_{H^\R}}$-compatible and $\varphi$ satisfies \eqref{eq:realhiggs},
which in this case is
\begin{equation}\label{eq2:realhiggs}
\sigma_{V}\otimes\sigma_{K_X}(\varphi)\,=\, \pm \sigma^*_X\varphi\, ,
\end{equation}
where $V\,=\,\EE_{H}(\liem)$.
The action of the gauge group $\HHH^\R$ of $\EE_{H^\R}$ preserves $\TTT^\pm$
and we define the moduli space $\Mg(G^\R,\sigma_X,\sigma_G,c,\pm)$
of $(\sigma_X,\,\sigma_G,\,c,\,\pm)$-real solutions to
(\ref{hitchin}) on $\EE_{H^\R}$ as the set of triples 
$(\sigma_{\EE_{H}},\,A,\,\varphi)\,\in\, \TTT^\pm$ with $(A,\,\varphi)$ satisfying 
\eqref{hitchin} modulo the gauge group $\HHH^\R$.

Theorem \ref{th:HKreal} can now be reformulated as follows.

\begin{The}\label{th2:HKreal} 
Fix the topological class of $\EE_H$ to be 
$d\,\in\,\pi_1(H)$, then \eqref{HK-Higgs-bundles} 
restricts to give a homeomorphism
$$
\cM_d(G^\R,\sigma_X,\sigma_G,c,\pm)\,\cong\,
\Mg(G^\R,\sigma_X,\sigma_G,c,\pm)\, .
$$
\end{The}

\subsection{Real Higgs bundles and compatible representations of the
orbifold fundamental group}\label{Subsection4.3}

Let $x\,\in\, X$ and $\pi_1(X,\,x)$ be the fundamental group of $X$.
A \textbf{representation} of $\pi_1(X,\,x)$ in
$G^\R$ is a homomorphism $\rho\,\colon\, \pi_1(X,\,x) \,\longrightarrow\, G^\R$.
The group $G^\R$ acts on the set $\Hom(\pi_1(X,\,x),\,G^\R)$ of all such 
homomorphisms by conjugation:
\[
(g \cdot \rho)(\gamma) \,= \,g \rho(\gamma) g^{-1}
\]
for $g\,\in\, G^\R$, $\rho \,\in\, \Hom(\pi_1(X,\,x),\,G^\R)$ and $\gamma\,\in\,
\pi_1(X,\,x)$. If we restrict the action to the subspace
$\Hom^+(\pi_1(X,\,x),\,G^\R)$ consisting of reductive
representations, the orbit space is Hausdorff. By a \textbf{reductive representation} we mean one
that composed with the adjoint representation in the Lie algebra
of $G^\R$ decomposes as a sum of irreducible representations.
Define the
{\bf moduli space of representations} or {\bf character variety} of $\pi_1(X,\,x)$ in $G^\R$ to be
the orbit space
\[
\mathcal{R}(G) \,=\, \Hom^{+}(\pi_1(X,\,x),\,G) / G\, . 
\]

The moduli space $\calR(G^\R)$ is independent of the choice of base point $x\in X$ and has the
structure of an algebraic variety \cite{Ri}.

Given a representation $\rho\,\in\, \Hom(\pi_1(X,\,x),\,G^\R)$
there is an associated flat $G^\R$-bundle on
$X$, and conversely, the holonomy of a flat $G^\R$-bundle on $X$ defines
a representation $\rho\,\in\, \Hom(\pi_1(X,\,x),\,G^\R)$.
More precisely, there is a bijection between the set of equivalence 
classes of representations $\Hom(\pi_1(X,\,x),\,G^\R) /G^\R$ and the set 
of equivalence classes of flat $G^\R$-bundles, which in turn is 
parameterized by the cohomology set $H^1(X,\,G^\R)$. 

Given a representation $\rho\in \Hom(\pi_1(X,\,x),\,G^\R)$ there is a topological invariant taking values in
$\pi_1(G^\R)\,=\,\pi_1(H^\R)$. Fixing $d\,\in\, \pi_1(G^\R)$
we can consider the subvariety 
$\calR_d(G^\R)\,\subset\, \calR(G^\R)$ consisting of those representations with topological class $d$.

The {\bf non-abelian Hodge correspondence} states that there is a homeomorphism
\begin{equation}\label{na-hodge}
\mathcal{R}_d(G^\R)\,\cong\, \mathcal{M}_d(G^\R)\, . 
\end{equation}

The proof of (\ref{na-hodge}) is given by combining (\ref{HK-Higgs-bundles}) 
and Donaldson--Corlette's existence theorem of harmonic metrics.
To explain this, let
$\EE_{G^\R}$ be a $C^\infty$ principal $G^\R$-bundle over $X$ with
fixed topological class $d\in\pi_1(G^\R)\,=\,\pi_1(H^\R)$. 
Fix a reduction $\EE_{H^\R}$ of the structure group $h$ of $\EE_{G^\R}$ to $H^\R$.
We have that
$$
\EE_{G^\R}(\lieg^\R) \,=\,\EE_{H^\R}(\lieh^\R)\oplus\EE_{H^\R}(\liem^\R)\,.
$$
 
The covariant derivative $D$ on 
$\EE_{G^\R}(\lieg^\R)$ of a connection on $\EE_{G^\R}$ decomposes uniquely as
\begin{equation}\label{connection-decomposition}
D\,=\,d_A + \psi\, ,
\end{equation}
where $A$ is a connection on $\EE_{H^\R}$ and $d_A$ is its
covariant derivative, and
$\psi\,\in\, \Omega^1(X,\,\EE_{H^\R}(\liem^\R))$. Let
$F_A$ be the curvature of $A$.
Consider the following set of equations for the pair $(A,\,\psi)$:
\begin{equation}\label{harmonicity}
\begin{array}{l}
F_A +\frac{1}{2}[\psi,\,\psi]\,=\, 0\\
d_A\psi\,=\,0\\
d_A^\ast\psi\,=\,0\, .
\end{array}
\end{equation}
These equations are invariant under the action of $\HHH^\R$, the gauge group of
$\EE_{H^\R}$. The theorem of Corlette \cite{corlette} 
(see Donaldson \cite{donaldson} for $G^\R=\SL(2,\CC)$), says that there 
is a homeomorphism
\begin{equation}\label{corlette}
\{\mbox{Reductive flat connections $D$ on $\EE_{G^\R}$}\}/\GGG^\R\,\cong\,
\{(d_A,\,\psi)\;\;\mbox{satisfying}\;\;
(\ref{harmonicity})\}/\HHH^\R.
\end{equation}
Recall that a flat connection is said to be {\bf reductive} if the 
holonomy representation is reductive.

To complete the argument, leading to (\ref{na-hodge}), we just need
(\ref{HK-Higgs-bundles}) and the simple fact that the correspondence 
$(A,\,\varphi)\,\longmapsto\, (A,\,\psi\,:=\,\varphi-\tau(\varphi))$
defines a homeomorphism
\begin{equation}\label{hitchin-harmonicity}
\{(A,\varphi)\;\;\mbox{satisfying}\;\;
(\ref{hitchin})\}/\HHH^\R\cong
\{(A,\,\psi)\;\;\mbox{satisfying}\;\;
(\ref{harmonicity})\}/\HHH^\R\, .
\end{equation}

The first two equations in (\ref{harmonicity}) are equivalent
to the flatness of $D\,=\,d_A+\psi$, and (\ref{corlette})
simply says that in the $\GGG^\R$-orbit of a reductive flat connection
$D_0$ on $\EE_{G^\R}$ we can find a flat connection $D\,=\,g(D_0)$ with $g\in \GGG^\R$ 
such that if we write $D\,=\,d_A+\psi$, the
additional condition $d_A^\ast\psi\,=\,0$ is satisfied, where the operator $d_A^\ast$ is defined 
using the Hodge $\ast$-operator given by the complex structure of $X$. This can be interpreted
more geometrically in terms of the reduction $g(h)$ of $\EE_{G^\R}$
to an $H^\R$-bundle obtained by the action of $g$ on the initial reduction $h$.
The equation
$d_A^\ast\psi\,=\,0$ is equivalent to the {\bf harmonicity} of the section of the $G^\R/H^\R$-bundle 
$\EE_{G^\R}(G^\R/H^\R)$ defined by $g(h)$.

If, as in Section \ref{Subsection4.1}, we are equipped 
with conjugations $\sigma_X$ and $\sigma_G$, $c\,\in\, Z_\iota^{\sigma_G}$, and
$\EE_H$ is a $C^\infty$ $H$-bundle, we can consider a
$(\sigma_X,\sigma_G,c)$-real structure $\sigma_{\EE_{H}}$ on $\EE_{H}$.
If we fix also a reduction $h$ of structure group $\EE_{H^\R}$
chosen to be $\sigma_{\EE_{H}}$-compatible, then $\sigma_{\EE_{H}}$ will restrict to
$\sigma_{\EE_{H^\R}}$ on $\EE_{H^\R}$, and since $\sigma_G$ preserves $\liem^\R$,
we also have a restriction $\sigma_{\EE_{H^\R}(\liem^\R)}\,:\,\EE_{H^\R}(\liem^\R)\,
\longrightarrow\, \EE_{H^\R}(\liem^\R)$.
Let $\EE_{G^\R}$ the $G^\R$-bundle obtained from $\EE_{H^\R}$ by extension of 
structure group. Since $\sigma_G$ preserves $G^\R$, we can extend $\sigma_{\EE_{H^\R}}$
to a map $\sigma_{\EE_{G^\R}}\,:\, \EE_{G^\R}\,\longrightarrow\, \EE_{G^\R}$, and we also have
$\sigma_{\EE_{G^\R}(\lieg^\R)}\,:\, \EE_{G^\R}(\lieg^\R)\,\longrightarrow\, \EE_{G^\R}(\lieg^\R)$.

Let $D$ be a connection on $\EE_{G^\R}$. Recall that this can be regarded as a 
map $D\,:\, T\EE_{G^\R}\,\longrightarrow\, \lieg^\R$. The compact conjugation $\tau\,:\,\lieg
\,\longrightarrow\,\lieg$ 
preserves $\lieg^\R$ and hence we can consider the connection 
$\tau(D)\,:=\,\tau\circ D$.
Let $\PPP^\pm$ be the set of pairs $(\sigma_{\EE_{G^\R}},\,D)$, where
$\sigma_{\EE_{G^\R}}$ is as above and $D$ is a connection on $\EE_{G^\R}$ such 
that the connection $\tau^{\frac{1}{2}\mp\frac{1}{2}}(D)$ is 
$\sigma_{\EE_{G^\R}}$-compatible. Abusing notation, let $D$ denote also the 
covariant derivative defined by $D$ on $\EE_{G^\R}(\lieg^\R)$. 
From (\ref{connection-decomposition}) we have the decomposition
$D\,=\,d_A+\psi$, where $A$ is a connection on $\EE_{H^\R}$ and 
$\psi\in \Omega^1(X, \EE_{H^\R}(\liem^\R))$. Since $\tau(A)\,=\,A$ and 
$\tau(\psi)\,=\,-\psi$, the set of pairs $\PPP^\pm$ can be identified 
with the set of triples $(\sigma_{\EE_{G^\R}},\,A,\,\psi)$, where
$A$ is $\sigma_{\EE_{H^\R}}$-compatible and
$\sigma_{\EE_{G^\R}(\liem^\R)} (\psi)\,=\,\pm \sigma_X^*(\psi)$.
Setting $\psi\,:=\,\varphi-\tau(\varphi)$, it is immediate that $\PPP^\pm$
is then in bijection with the set 
of triples $(\sigma_{\EE_{H}},\,A,\,\varphi)\,\in\, \TTT^\pm$ considered in Section 
\ref{Subsection4.1}.
The uniqueness of the harmonic section provided by the Donaldson--Corlette theorem \eqref{corlette} (recall
that $G^\R$ is semisimple) implies that this section is 
fixed by $\sigma_{\EE_{G^\R}}$, or more precisely, by the corresponding map defined on the space of sections
of $\EE_{G^\R}(G^\R/H^R)$. As a consequence of this and 
(\ref{hitchin-harmonicity}) we have the following.

\begin{Pro}\label{corlette-pm}
The bijection (\ref{corlette}) restricts to give a bijection
$$
\{(\sigma_{\EE_{G^\R}},D)\in\PPP^\pm\;\; \mbox{with $D$ reductive flat}\}/\GGG^\R 
\cong \Mg(G^\R,\sigma_X,\sigma_G,c,\pm).
$$
\end{Pro}

We now need to identify 
$\{(\sigma_{\EE_{G^\R}},D)\in\PPP^\pm\;\; \mbox{with $D$ reductive flat}\}/\GGG^\R$ in terms of representations of
the fundamental group of $X$. To do this, fix a point $x\,\in\, X,$ such that $\sigma_{X}(x)\,\neq\, x$. 
The \textbf{orbifold fundamental group} $\Gamma(X,\,x)$ of $(X,\,\sigma_{X})$ is, as a set, the
disjoint union of
$\pi_{1}(X,x)$ and $$\Path(X,\,x)\,:=\,\{\text{Homotopy classes of paths }~
\gamma\,:\,[0,\,1]\,\longrightarrow \,X\, \mid \,\gamma(0)\,=\,x, ~\gamma(1)\,=\,\sigma_{X}(x)\},$$
with the composition defined by $\gamma_{2}\gamma_{1}\,=\,
\sigma_{X}^{q}(\gamma_{2})\circ\gamma_{1}\, ,$ where $$
q=\begin{cases}
0& \text{if $\gamma_1\in \pi_1(X,\,x)$}\\
1& \text{if $\gamma_1\in \Path(X,\,x)$}\,.
\end{cases}
$$The function $q$ induces the short exact sequence
\begin{equation}\label{hq}
\xymatrix{
0\ar[r] &\pi_1(X,\,x)\ar[r]^{i} &\Gamma(X,\, x)\ar[r]^{q} &\mathbb{Z}/2\mathbb{Z}\ar[r]&0}\, ,
\end{equation}
where $i$ denotes the inclusion of groups.

Let $c\,\in\, Z^{\sigma_{G}}_{2}(H)\cap \Ker(\iota)\cap Z(G^\R)$, where recall
$Z(H)$ and 
$Z(G^\R)$ are the centers of $H$ and $G^\R$ respectively, and $Z^{\sigma_{G}}_{2}(H)$ is the subgroup of elements of order 2 in $Z(H)$ invariant under $\sigma_G$.

Let $\widehat{G}^\R_\pm\,=\,\widehat{G}^\R_\pm(\sigma_G,c)$ be the group whose underlying set is $G^\R\times(\mathbb{Z}/2\mathbb{Z})$ and
the group operation on it is given by
$$(g_1,\,e_1)(g_2,\,e_2)\,=\,(g_{1}(\sigma_{G}\tau^{\frac{1}{2}\mp \frac{1}{2}})^{e_1}(g_2) c^{e_{1}+ e_{2}},
\,e_{1}+e_{2})\, .$$

A representation $\widehat{\rho}\,:\,
\Gamma(X,\, x)\,\longrightarrow\, \widehat{G}^\R_\pm$ 
is called $(\sigma_X,\,\sigma_{G},\,c,\,\pm)$-\textbf{compatible} 
if it is an extension of a representation 
$\rho\,:\, \pi_1(X,\, x)\,\longrightarrow\, G^\R$ 
fitting in a commutative diagram of homomorphisms
\begin{equation}
\xymatrix{
0\ar[r] &\pi_1(X,\,x)\ar[r]^{i}\ar[d]^{\rho} &\Gamma(X,\, x)\ar[r]^{q}\ar[d]^{\widehat{\rho}} &\mathbb{Z}/2\mathbb{Z}\ar[r]\ar[d]^{\Id}&0\\
0\ar[r] &G^\R\ar[r]^{i'} &\widehat{G}^\R_\pm\ar[r]^{q'} &\mathbb{Z}/2\mathbb{Z}\ar[r]&0\,,
}
\end{equation}
where $i,\, i'$ are the inclusion maps and $q$ and $q'$ are the corresponding projections.

Let $\Rr(G^\R,\,\sigma_X,\,\sigma_{G},\,c,\,\pm)$ be the variety
consisting of $G^\R$-conjugacy classes of $(\sigma_X,\,\sigma_{G},\,c,\,\pm)$-compatible 
representations 
$\widehat{\rho}\,:\,
\Gamma(X,\, x)\,\longrightarrow\, \widehat{G}^\R_\pm$ 
whose restriction to $\pi_1(X,\,x)$ is reductive, that is,
its conjugacy class is an element in $\calR(G^\R)$.
We have the following.

\begin{Pro}\label{holonomy-pm}
The holonomy representation defines a bijection
$$
\{(\sigma_{\EE_{G^\R}},D)\,\in\,\PPP^\pm\;\; \mbox{with $D$ reductive flat }\}/\GGG^\R \cong \Rr_d(G^\R,\sigma_X,\sigma_{G},c,\pm).
$$
\end{Pro}

\begin{proof}
Recall that the holonomy representation $\rho \,:\,\pi_{1}(X,x)\,\longrightarrow\, G^\R$ induced by
a flat connection $D$ on $\EE_{G^\R}$ is 
defined as follows: let $\gamma\,\in\,\pi_{1}(X,x)$, there is a lift 
$\widetilde{\gamma}\,:\,[0,\,1]\,\longrightarrow\, \EE_{G^\R}$ of $\gamma$ 
such that $D(\widetilde{\gamma}_{*}\frac{\partial}{\partial t})=0.$ Let $e\,=\,\widetilde{\gamma}(0).$
If $\widetilde{\gamma}(1)\,=\,eg$, then $\rho(\gamma)\,:=\,g.$

Let $(\sigma_{\EE_G^\R},\,D)\,\in\, \PPP^\pm$ where $D$ is reductive flat connection.
 We define a homomorphism 
$\widehat{\rho}\,=\,\hol(D)\,:\,\Gamma(X,x)\,\longrightarrow\,\widehat{G}^\R_\pm $ 
 in the following way:
${\widehat{\rho}}|_{\pi_{1}(X,x)}\,:=\,\rho.$ Let $\gamma\,\in\,
\Gamma(X,\,x)\setminus \pi_1(X,\,x)$, and let $e$ be an element in the fiber 
$\EE_{G^\R}$ at $x$. 
Let $e'\,\in\, \EE_{G^\R}$ be the point obtained by parallel transport by $D$ 
of $e$ 
along $\gamma$. Let $g_{\gamma}\,\in\, G^\R$ be the element of the group such that 
\begin{equation}\label{eq:transporte}
e\,=\,f(e')\sigma_{G}(g_{\gamma})\,,
\end{equation}
where $f\,:\,\EE_{G^\R}\,\longrightarrow\, \sigma_{X}^{*}\sigma_{G}(\EE_{G^\R})$ 
is the isomorphism of
principal $G^\R$-bundles induced by $\sigma_{\EE_{G^\R}}.$ 
We define $$\widehat{\rho}(\gamma)\,:=\,g_{\gamma}\, ,$$ for every 
$\gamma\,\in\,\Gamma(X,\,x)\setminus \pi_{1}(X,\,x).$ 
Following the same arguments as
in the proof of \cite[Proposition 4.4]{BGH:98}, one
shows that $\widehat{\rho}$ is a $(\sigma_X,\,\sigma_{G},\,c,\,\pm)$-compatible representation.
Equivalence classes of $(\sigma_{\EE_G^\R},\,D)\in \PPP^\pm$ where $D$ is reductive flat connection correspond to equivalence classes of 
$(\sigma_X,\sigma_{G}, c, \pm)$-compatible representations
since $\hol(D)$ satisfies that $\hol(D)(eg)\,=\, g\hol(D)(e)g^{-1}$
for all $g\,\in\, G^\R$ and for all $e$ in the fiber of $\EE_{G^\R}$ at the point $x$. 

Conversely, let $\widehat{\rho}\,:\,
\Gamma(X,\, x)\,\longrightarrow\, \widehat{G}^\R_\pm$ 
be a $(\sigma_X,\sigma_{G},c,\pm)$-compatible representation. 
The induced representation $\rho\,:\,\pi_1(X,\,x)\,\longrightarrow\, G^\R$ 
corresponds to a flat $G^\R$-bundle $(\EE_{G^\R},\,D)$. A map 
$f\,:\,\EE_{G^\R}\,\longrightarrow\, 
\sigma_{X}^{*}\sigma_{G}(\EE_{G^\R})$ can be constructed using 
(\ref{eq:transporte}) and it can be 
extended to the other fibers as is done in \cite[p. 18]{BGH:98}. From \cite[Proposition 
4.3]{BGH:98}, the morphism $f$ defines a pair
$(\sigma_{\EE_{G^\R}},\,D)\,\in\, \PPP^\pm$. The $\sigma_{\EE_{G^\R}}$-compatibility of 
$\tau^{\frac{1}{2}\mp\frac{1}{2}}(D)$ follows from the 
$(\sigma_X,\,\sigma_{G},\,c,\,\pm)$-compatibility of $\widehat{\rho}$.
\end{proof}

Combining Propositions \ref{holonomy-pm} and \ref{corlette-pm} and Theorem
\ref{th2:HKreal} we have now the following non-abelian Hodge correspondence for real $G^\R$-Higgs bundles. 

\begin{The}\label{th:non-abelian}
There is a homeomorphism 
$$
\Rr_d(G^\R,\,\sigma_X,\,\sigma_{G},\,c,\,\pm)\,\cong\, \Mm_d(G^\R,\,\sigma_{X},\,\sigma_{G},\,c,\,\pm)\, .
$$
\end{The}

\begin{Rem}
Theorem \ref{th:non-abelian} generalizes the result
in \cite{BGH:99} for a complex reductive Lie group $G$, since $G$ can be viewed as a real form of $G\times G,$
as mentioned in the introduction.
\end{Rem}

%%%%%%%%%%%%%%%%%%%%%%%%%%%%%%%%%%%%%%%%%%%%%%%%%%%%%%%%%%%%%%%
\section{Involutions of moduli spaces} \label{Section5}
%%%%%%%%%%%%%%%%%%%%%%%%%%%%%%%%%%%%%%%%%%%%%%%%%%%%%%%%%%%%%%%

Consider the same setup and notation of Section \ref{Section4}.

%%%%%%%%%%%%%%%%%%%%%%%%%%%%%%%%%%%%%%%%%%%%%%%%%%%%%%%
\subsection{Involutions of Higgs bundle moduli spaces}
%%%%%%%%%%%%%%%%%%%%%%%%%%%%%%%%%%%%%%%%%%%%%%%%%%%%%%%

Let $\Mm(G^\R)$ be the moduli space of polystable $G^\R$-Higgs bundles 
over $X$. The $(\sigma_X,\sigma_G,c,\pm)$-real $G^\R$-Higgs bundles studied in
Section \ref{Section4} appear in a natural way as fixed points 
of certain involutions on the moduli space $\Mm(G^\R)$, in a similar way as they do in the case when $G^\R$ is complex, as studied in \cite{BGP}, whose approach we follow closely. Recall that $\sigma_X$ is a conjugation of $X$, $\sigma_G$ is a conjugation of $G$ satisfying the 
conditions given at the beginning of Section \ref{Section4}, and 
$c\,\in \,Z_\iota^{\sigma_G}$, where $Z_\iota^{\sigma_G}$ is defined by 
(\ref{eq:intersection}).

Let $(E,\,\varphi)$ be a $G^\R$-Higgs bundle on $X$. Let $\sigma_G(E)$ be the
$C^\infty$ principal $H$-bundle on $X$ obtained by extending the structure 
group of $E$ using the conjugation $\sigma_G\,:\,H\,\longrightarrow\, H$. Since $\sigma_X$ is 
antiholomorphic, the pullback
$\sigma_X^*\sigma_G(E)$ is a holomorphic $H$-bundle over $X$. 

Let
$$
\widetilde{\sigma_G}\, :\, E({\mathfrak m})\, \longrightarrow\,E({\mathfrak m})
$$
be the conjugate linear isomorphism that sends the equivalence class of
any $(e\, , v)\, \in\, E\times {\mathfrak m}$ to the equivalence class of 
$(e\, , \sigma_G (v))$. Let $\sigma_G(\varphi)$ be the
$C^\infty$ section of $E({\mathfrak m})\otimes K_X$ defined by 
$\widetilde{\sigma_G}$
and the antiholomorphic involution $\sigma_{K_X}: \,K_X\, \longrightarrow\, K_X$ induced by $\sigma_X$. We have the following involutions.

\begin{equation}\label{eq:ecuacionesp122}
\mapnormal{\Mm(G^\R)}{\Mm(G^\R)}{(E,\varphi)}{(\sigma_{X}^{*}\sigma_{G}(E),\pm\sigma_{X}^{*}\sigma_{G}(\varphi))\,.}{\iota_{\Mm}(\sigma_{X},\sigma_{G})^{\pm}}
\end{equation}

As a consequence of Corollary \ref{Cor:realpolypoly}, there is a forgetful map
$$\mapnormal{\Mm(G^\R,\sigma_{X},\sigma_{G},c, \pm)}{\Mm(G^\R)}{(E,\varphi, \sigma_{E})}{(E,\varphi)\,.}{f_{\Mm}}$$
We denote by $\widetilde{\Mm}(G^\R,\sigma_{X},\sigma_{G},c,\pm)$ the image of $f_{\Mm}$.

\begin{Pro}\label{pro:pseudoreales} The fixed points of $\iota_{\Mm}(\sigma_{X},\sigma_{G})^{\pm}$ and the moduli spaces 
of polystable $(\sigma_X,\,\sigma_G,\,c,\,\pm)$-real $G^\R$-Higgs bundles are 
related as follows.
\begin{enumerate}
\item $$\Mm(G^\R)^{\iota_{\Mm}(\sigma_{X},\sigma_{G})^{\pm}}\,\supseteq\, 
\bigcup_{c\in Z_\iota^{\sigma_G}}
\widetilde{\Mm}(G^\R,\sigma_{X},\sigma_{G},c,\pm)\,.$$

\item For $g(X)\geq 2$ and if we restrict the involution to the subvariety $\Mm_{\sm}(G^R)\,\subset\,\cM(G^\R)$ of stable
and simple $G^\R$-Higgs bundles, then
$$\Mm(G^\R)_\sm^{\iota_{\Mm}(\sigma_{X},\sigma_{G})^{\pm}}
\subseteq \bigcup_{c\in Z_\iota^{\sigma_G}}\widetilde{\Mm}(G^\R,\sigma_{X},\sigma_{G},c,\pm)\,.$$
\end{enumerate}
\end{Pro}

\begin{proof}
(1) If $(E,\varphi)$ 
is the image of $(E,\varphi,\sigma_E)\,\in\, \Mm(G^\R,\sigma_{X},\sigma_{G},c,\pm) $, then
there is a holomorphic isomorphism of bundles $f\,:\, E
\,\longrightarrow\, \sigma_{X}^{*}\sigma_{G}(E)$ induced by $\sigma_E$, and this
defines a map $\widetilde{f}\,:\,E(\liem)\otimes K_{X}\,\longrightarrow\, 
\sigma_{X}^{*}\sigma_{G}(E)(\liem)\otimes K_{X}$. We thus have

$$
\iota_{\Mm}(\sigma_{X},\sigma_{G})^{\pm}(E,\varphi)\cong (f(E),
\pm \widetilde{f}(\varphi))\cong (E,\varphi)\,.
$$
	 
(2) Let $(E,\,\varphi)\,\in\, \Mm(G^\R)_\sm^{\iota_{\Mm}(\sigma_{X},\sigma_{G})^{\pm}}$. 
There is an isomorphism $f\,:\, E\,\longrightarrow\, \sigma_{X}^{*}\sigma_{G}(E)\,,$
such that $(f(E),\,\pm \widetilde{f}(\varphi))\,\cong\, (E,\,\varphi)\,.$ The composition $\sigma_{X}^{*}\sigma_{G}(f) \,\circ\,f$ belongs to
the group $\Aut(E,\,\varphi)$ of automorphisms of $(E,\,\varphi)$. Since $(E,\,\varphi)$ is assumed to be
simple, this group coincides with $Z(H)\cap \Ker(\iota)$. Let 
$c\,:=\,\sigma_{X}^{*}\sigma_{G}(f)\circ f \,\in\, Z(H)\cap \Ker(\iota).$ 
Since $f$ commutes with $\sigma_{X}^{*}\sigma_{G}(f)\circ f$, we have
that $\sigma_{G}(c)\,=\,c.$ Therefore, $c\,\in\, Z_{2}^{\sigma_{G}}\cap \ker(\iota)$,
and $f$ defines a $(\sigma_{X},\sigma_{G},c)$-real structure $\sigma_{E}$ on $E$. 
For the real structure $\sigma_{E(\liem)}\otimes \sigma_{K_X}$ induced by $\sigma_{E}$ and $\sigma_{X}$ 
we have that $\sigma_{E(\liem)}\otimes \sigma_X(\varphi)\,=\,\pm\varphi$, since by hypothesis 
$\varphi\,\cong\, \pm\widetilde{f}(\varphi).$
\end{proof}

%%%%%%%%%%%%%%%%%%%%%%%%%%%%%%%%%%%%%%%%%%%%%%%
\subsection{Involutions of character varieties}
%%%%%%%%%%%%%%%%%%%%%%%%%%%%%%%%%%%%%%%%%%%%%%%

We study now the involutions of the character variety $\calR(G^\R)$ corresponding to the involutions
\eqref{eq:ecuacionesp122} of 
$\cM(G^\R)$ via the non-abelian Hodge correspondence (\ref{na-hodge}). This generalizes the case in which
$G^\R$ is complex treated in \cite{BGP}, which again we follow closely.

Fix a point $x\,\in\, X$. The involution $\sigma_X$ of $X$ produces an
involutive isomorphism
$$
{(\sigma_X)}_*\, :\, \pi_1(X, x)\, \longrightarrow\, \pi_1(X, \sigma_X(x))\, .
$$
This in turn gives an involution 
\begin{equation}\label{bhm}
{(\sigma_X)}_*\, :\, \Hom^{+}(\pi_1(X,x),G^\R) / G^\R \, \longrightarrow\, \Hom^{+}(\pi_1(X,
\sigma_X(x)),G^\R) / G^\R \, ,
\end{equation}
which, abusing notation, we are also denoting by $(\sigma_X)_*$.
As mentioned above, $\calR(G^\R) \,=\, \Hom^{+}(\pi_1(X,x),G^\R) / G^\R$ is independent
of the choice of the base point. 

Let $\sigma$ be an involution of $G^\R$. This defines an involution (denoted also
by $\sigma$)
$$
\sigma\, : \, \calR(G^\R)\, \longrightarrow\,
\calR(G^\R)
$$
given by $\rho\, \longmapsto\, \sigma\circ\rho$.
In other words, $\sigma$ sends a homomorphism $\rho\, :\, \pi_1(X,x)\,\longrightarrow\, G^\R$
to the composition
$$
\pi_1(X,x)\,\stackrel{\rho}{\longrightarrow}\, G^\R\,\stackrel{\sigma_G}{\longrightarrow}
\, G^\R\, .
$$
Clearly this involution commutes with the involution $(\sigma_X)_*$ in \eqref{bhm}. Therefore,
$\sigma\circ\ {(\sigma_X)}_*$ is also an involution. 

One has the following.

\begin{Pro}\label{rep-involutions}
Let 
$\iota_{\Rr}(\sigma_{X},\,\sigma_{G})^{\pm}$ be the map given by
$$
 \mapnormal{\Rr(G^\R)}{\Rr(G^\R)}{\rho}{\sigma_{G} \circ \tau^{\frac{1}{2}\mp\frac{1}{2}}\circ\rho \circ
(\sigma_{X})_{*}}{\iota_{\Rr}(\sigma_{X},\sigma_{G})^{\pm}}.
$$
Then $\iota_{\Rr}(\sigma_{X},\sigma_{G})^{\pm}$ is an involution of $\calR(G^\R)$ and the following diagram commutes:
\[
\xymatrix{
\Mm(G^\R) \ar[rr]^{\cong} \ar[d]_{\iota_{\Mm}(\sigma_{X},\sigma_{G})^{\pm}} & & \Rr(G^\R) \ar[d]^{\iota_{\Rr}(\sigma_{X},\sigma_{G})^{\pm}}
\\
\Mm (G^\R)\ar[rr]^{\cong} & & \Rr(G^\R).
}
\]
\end{Pro}

\begin{proof}
The fact that ${\iota_{\Rr}(\sigma_{X},\sigma_{G})^{\pm}}$ is an involution follows from the previous discussion. The commutativity
of the diagram follows immediately from the construction of the non-abelian Hodge correspondence map $\cM(G^\R)\to \calR(G^\R)$.
Recall from Section \ref{Subsection4.1} that if $(E,\varphi)$ is a polystable $G^\R$ Higgs bundle, one associates to it the flat 
$G^\R$-connection $D\,=\,d_A+\varphi-\tau(\varphi)$, where the pair $(A,\,\varphi)$ solves the Hitchin equations (\ref{hitchin}). The result
follows now from the properties of the holonomy map associating to $D$ a representation of $\pi_1(X,\,x)$ in $G^\R$.
\end{proof}

Let $\widetilde{\calR}(G^\R,\sigma_X,\sigma_{G},c,\pm)$ be the image in 
$\calR(G)$ of the map defined by restricting $\widehat{\rho}\,\in \,\calR(G^\R,\sigma_X,\sigma_{G},c,\pm)$ to
$\pi_1(X,\,x)$. Notice that the homeomorphism in
Theorem \ref{th:non-abelian} defines a homeomorphism 
between $\widetilde{\calR}(G^\R,\,\sigma_X,\,\sigma_{G},\,c,\,\pm)$
and $\widetilde{\cM}(G^\R,\sigma_X,\sigma_{G},c,\pm)$, which is indeed the restriction of the homeomorphism $\calR(G^\R)\cong \cM(G^\R)$ given by the non-abelian Hodge correspondence.
As a corollary of Propositions \ref{rep-involutions} and \ref{pro:pseudoreales}, and Theorem \ref{th:non-abelian} we have the following.

\begin{Pro} The fixed points of $\iota_{\calR}(\sigma_{X},\sigma_{G})^{\pm}$ and the moduli spaces 
of $(\sigma_X,\,\sigma_G,\,c,\,\pm)$-compatible representations of 
$\Gamma(X,\, x)$ in $\widehat{G}^\R_\pm$ 
are related as follows.
\begin{enumerate}
\item $$\calR(G^\R)^{\iota_{\calR}(\sigma_{X},\sigma_{G})^{\pm}}\supseteq 
\bigcup_{c\in Z_\iota^{\sigma_G}\cap Z(G^\R)}
\widetilde{\calR}(G^\R,\sigma_{X},\sigma_{G},c,\pm)\,.$$

\item For $g(X)\geq 2$ and if we restrict the involution to the subvariety $\calR_{\irr}(G^R)\,\subset\,\calR(G^\R)$ of 
irreducible representations, then
$$\calR_\irr(G^\R)^{\iota_{\calR}(\sigma_{X},\sigma_{G})^{\pm}}
\subseteq \bigcup_{c\in Z_\iota^{\sigma_G}\cap Z(G^\R)}\widetilde{\calR}(G^\R,\sigma_{X},\sigma_{G},c,\pm)\,.$$
\end{enumerate}
\end{Pro}

\begin{Rem}
Recall that a representation in $\calR(G^\R)$ is said to be {\bf irreducible} if the centralizer of the image of $\rho$ in $G^\R$ coincides with the center of $G^\R$. Under the non-abelian Hodge correspondence (\ref{na-hodge}) the subvariety of irreducible representations is in bijection
with the subvariety of stable and simple Higgs bundles (see \cite{GGM:21}).
\end{Rem}

\section*{Acknowledgements}

We thank the referees for their helpful comments to improve the exposition.
The first author is supported by a J. C. Bose Fellowship.
The third author was partially supported by the Spanish MINECO under ICMAT Severo Ochoa project No. SEV-2015-0554, and under grant No. MTM2013-43963-P

\end{document}